\documentclass[12pt]{article}
\usepackage{amsthm,amsmath,amssymb}

\usepackage{graphicx}

\usepackage[colorlinks=true,citecolor=black,linkcolor=black,urlcolor=blue]{hyperref}

\theoremstyle{plain}
\newtheorem{theorem}{Theorem}[section]
\newtheorem{lemma}[theorem]{Lemma}

\newtheorem{proposition}[theorem]{Proposition}

\theoremstyle{definition}
\newtheorem{definition}[theorem]{Definition}
\newtheorem{example}[theorem]{Example}

\theoremstyle{remark}
\newtheorem{remark}[theorem]{Remark}

\title{\bf Hopf algebra of building sets}

\author{Vladimir Gruji\'c\thanks{Authors are supported by Ministry of Science of Republic of Serbia, project
174034.}\\
\small Faculty of Mathematics\\[-0.8ex]
\small Belgrade University\\[-0.8ex]
\small Serbia\\
\small\tt vgrujic@matf.bg.ac.rs\\
\and
Tanja Stojadinovi\'c\\
\small Faculty of Mathematics\\[-0.8ex]
\small Belgrade University\\[-0.8ex]
\small Serbia\\
\small\tt tanjas@matf.bg.ac.rs}

\date{%\dateline{Jan 1, 2012}{Jan 2, 2012}\\
}

\begin{document}
\maketitle

\begin{abstract}
The combinatorial Hopf algebra on building sets $BSet$ extends the
chromatic Hopf algebra of simple graphs. The image of a building
set under canonical morphism to quasi-symmetric functions is the
chromatic symmetric function of the corresponding hypergraph. By
passing from graphs to building sets, we construct a sequence of
symmetric functions associated to a graph. From the generalized
Dehn-Sommerville relations for the Hopf algebra $BSet$, we define
a class of building sets called eulerian and show that eulerian
building sets satisfy Bayer-Billera relations. We show the
existence of the $\mathbf{c}\mathbf{d}-$index, the polynomial in
two noncommutative variables associated to an eulerian building
set. The complete characterization of eulerian building sets is
given in terms of combinatorics of intersection posets of
antichains of finite sets.

  % keywords are optional
\end{abstract}

\section{Introduction}

Many combinatorial objects may be endowed with a Hopf algebra
structure. The best known examples are Rota's Hopf algebra of
finite graded posets $\cite{JR}$ and the chromatic Hopf algebra of
simple graphs $\cite{WS1}$.

The theory of combinatorial Hopf algebras is developed in
$\cite{ABS}.$ In Section 2 we recall the basic definitions and
properties of combinatorial Hopf algebras and of quasi-symmetric
functions. The Hopf algebra of quasi-symmetric functions $QSym$ is
the terminal object in the category of combinatorial Hopf
algebras. It explains the ubiquity of quasi-symmetric functions as
generating functions in enumerative combinatorics.

The notion of a building set is originated in the work of De
Concini and Procesi $\cite{DCP}$ in the context of subspace
arrangement and developed by Feichtner and Sturmfels $\cite{FS}$
and Postnikov $\cite{P}$. The concept of a building set appears as
a combinatorial condition that polytopes from the certain class,
called nestohedra, are simple. An example of a building set is
provided by the collection of the vertex sets of connected
subgraphs in a given graph. Building sets are a kind of Whitney
systems. The Hopf algebra on Whitney systems is constructed in
$\cite{WS2}$. In Section 3, based on Schmitt's work, we introduce
the Hopf algebra of building sets $BSet$ that extends the
chromatic Hopf algebra $\mathcal{G}$ of simple graphs.

In Section 4 we define the chromatic symmetric function of a
building set as the image under the canonical morphism from
building sets to quasi-symmetric functions. To a building set
$\mathcal{B}=\mathcal{B}(\mathcal{C})$ is uniquely associated the
collection $\mathcal{C}_{\min}$ of minimal elements of the
generating collection $\mathcal{C}$, which is an antichain of
finite sets. The building sets may be seen as hypergraphs and
colorings of building sets are equivalent to colorings of
hypergraphs. The chromatic symmetric function of a building set
$\mathcal{B}(\mathcal{C})$ depends only on the associated
hypergraph $\mathcal{C}_{\min}$. In the setting of building sets,
we use the expansion of the chromatic symmetric function of a
hypergraph in the basis of the power sum symmetric functions,
given in $\cite{SS}$. The derived formulas for the induced
chromatic polynomial of a building set are analogues to the
classical Whitney's formulas for the chromatic polynomial of a
graph $\cite{W}$.

In Section 5 we construct a sequence of algebra morphisms
$\beta_n$ from graphs to building sets, that produces a sequence
of symmetric functions associated to graphs. Two numerical
invariants of graphs, the numbers of acyclic and of totally cyclic
orientations, arises from this construction.

Every combinatorial Hopf algebra $(\mathcal{H},\zeta)$ possesses
the canonical odd Hopf subalgebra $S_-(\mathcal{H},\zeta)$ on
which the character $\zeta$ is odd. This subalgebra is
characterized by certain canonical relations, called the
generalized Dehn-Sommerville relations. In the case of
$\mathcal{H}= QSym$, the generalized Dehn-Sommerville relations
are precisely the Bayer-Billera relations for flag $f-$vectors of
eulerian posets $\cite{BB}$. In Section 6, by analogy with
eulerian posets, we use the generalized Dehn-Sommerville relations
for the Hopf algebra $BSet$ to characterize a class of building
sets called eulerian. There is no analogues notion of eulerian
graphs, because no particular graph satisfies the generalized
Dehn-Sommerville relations for the chromatic Hopf algebra of
graphs. We derive that eulerian building sets satisfy the
Bayer-Billera relations.

The $\mathbf{c}\mathbf{d}-$ index of an eulerian poset is a
polynomial in noncommutative variables, firstly introduced by Fine
(see $\cite{BK}$). Its existence is equivalent to the
Bayer-Billera relations for the flag $f-$vector. The
$\mathbf{c}\mathbf{d}-$index is defined for elements of the
eulerian subalgebra of an infinitesimal Hopf algebra by a
construction given in $\cite{A}$. In Section 7, by analogy with
eulerian posets, we construct the $\mathbf{c}\mathbf{d}-$index of
eulerian building sets.

In Section 8, we consider how the algebraic condition of being
eulerian is related to combinatorics of antichains of finite sets.
We obtain the complete characterization and show that eulerian
building sets corresponds to clique complexes of chordal graphs.

\section{Combinatorial Hopf algebras}
\label{sec:1}

In this section we recall the basic definitions and properties of
combinatorial Hopf algebras, developed in $\cite{ABS}$.
Throughout, $n$ is a non-negative integer, $[n]$ denotes the set
$\{1,\ldots,n\}$ and $|X|$ denotes the cardinality of a finite set
$X$. A composition $\alpha\models n$ is a sequence
$\alpha=(a_1,\ldots,a_k)$ of positive integers with
$a_1+\cdots+a_k=n$. A partition $\lambda\vdash n$ is a multiset
$\lambda=\{l_1,\ldots,l_k\}$ such that $l_1+\cdots+ l_k=n$.

A combinatorial Hopf algebra $(\mathcal{H},\zeta)$ is a graded
connected Hopf algebra $\mathcal{H}$ over a field $\mathbb{K}$
equipped with a multiplicative linear functional
$\zeta:\mathcal{H}\rightarrow\mathbb{K}$, called character. A
morphism of combinatorial Hopf algebras $(\mathcal{H}_1,\zeta_1)$
and $(\mathcal{H}_2,\zeta_2)$ is a morphism of graded Hopf
algebras $\phi:\mathcal{H}_1\rightarrow\mathcal{H}_2$ such that
$\zeta_2\circ\phi=\zeta_1$.

\paragraph{Characters}

Let $\mathbb{X}(\mathcal{H})$ be the set of characters of an
arbitrary Hopf algebra $\mathcal{H}$. The set
$\mathbb{X}(\mathcal{H})$, under the convolution product

$$\varphi\psi=m_\mathbb{K}\circ(\varphi\otimes\psi)\circ\Delta_\mathcal{H},$$
is a group with the unit $\epsilon_\mathcal{H}$ and the inverse
$\varphi^{-1}=\varphi\circ S_\mathcal{H},$ where
$\epsilon_\mathcal{H}$ and $S_\mathcal{H}$ are the counit and the
antipode of the Hopf algebra $\mathcal{H}$.

Let $\mathcal{H}_n$ be the homogeneous component of the grading
$n$ of a graded Hopf algebra $\mathcal{H}$. Denote by
$\varphi_n=\varphi|_{\mathcal{H}_n}$ the restriction of a
character $\varphi$ on the component $\mathcal{H}_n$. The
conjugate character $\overline{\varphi}$ is defined on homogeneous
elements by $\overline{\varphi}(h)=(-1)^{n}\varphi(h), \
h\in\mathcal{H}_n.$ A character $\varphi$ is said to be even if
$\varphi=\overline{\varphi}$ and it is said to be odd if
$\varphi^{-1}=\overline{\varphi}.$ Every character $\varphi$ on a
graded connected Hopf algebra decomposes uniquely as a product of
characters $\varphi=\varphi_+\varphi_-,$ with $\varphi_+$ even and
$\varphi_-$ odd $(\cite{ABS}, {\rm Theorem} \ 1.5.).$

The odd subalgebra $S_{-}(\mathcal{H},\zeta)$ of a combinatorial
Hopf algebra $(\mathcal{H},\zeta)$ is defined as the largest
graded subcoalgebra on which the character $\zeta$ is odd. If
$\phi:(\mathcal{H}_1,\zeta_1)\rightarrow(\mathcal{H}_2,\zeta_2)$
is a morphism of combinatorial Hopf algebras then

$$\phi(S_-(\mathcal{H}_1,\zeta_1))\subset
S_-(\mathcal{H}_2,\zeta_2).$$

For a character $\varphi$ and a composition
$\alpha=(a_1,\ldots,a_k)\models n$, denote by $\varphi_\alpha$ the
convolution product

\begin{equation}\label{convolution}
\varphi_{a_1}\cdots\varphi_{a_k}:\mathcal{H}
\stackrel{\Delta^{(k-1)}}\longrightarrow\mathcal{H}^{\otimes
k}\stackrel{proj}\longrightarrow\mathcal{H}_{a_1}\otimes\cdots\otimes\mathcal{H}_{a_k}
\stackrel{\varphi^{\otimes k}}\longrightarrow\mathbb{K}.
\end{equation}

\paragraph{Quasi-symmetric functions}

The basic reference for quasi-symmetric functions is $\cite{EC}$.
The algebra $QSym$ of quasi-symmetric functions is a graded
subalgebra of the algebra $\mathbb{K}[[x_1,x_2,\ldots]]$ of formal
power series of finite degree in countably variables. It is
linearly spanned by monomial quasi-symmetric functions

$$M_\alpha=\sum_{i_1<i_2<\cdots<
i_k}x_{i_1}^{a_1}x_{i_2}^{a_2}\cdots x_{i_k}^{a_k},$$where
$\alpha=(a_1,a_2,\ldots a_k)\models n$ is a composition of an
integer $n\in\mathbb{N}$. It is a graded Hopf algebra with
coproduct

$$\Delta(M_\alpha)=\sum_{\alpha=\beta\gamma}M_\beta\otimes
M_\gamma,$$where $\beta\gamma$ is the concatenation of
compositions $\beta$ and $\gamma$.

Let $\zeta:\mathbb{K}[[x_1,x_2,\ldots]]\rightarrow\mathbb{K}$ be
an algebra morphism defined on variables by $\zeta(x_1)=1$ and
$\zeta(x_i)=0$ for $i\neq 1$. The universal character $\zeta_Q$ on
$QSym$ is the restriction $\zeta_Q=\zeta|_{QSym}$. It is
determined on the monomial basis by

$$\zeta_Q(M_\alpha)=\left\{\begin{array}{cc} 1, & \alpha=(n)
\ \mbox{or} \ ()\\ 0, & \mbox{otherwise}\end{array}\right..$$

\noindent One of the main results of $(\cite{ABS}, {\rm Theorem} \
4.1.)$ is that for an arbitrary combinatorial Hopf algebra
$(\mathcal{H},\zeta)$, there is a unique morphism of combinatorial
Hopf algebras $\Psi:(\mathcal{H},\zeta)\longrightarrow
(QSym,\zeta_Q),$ which is defined on homogeneous elements
$h\in\mathcal{H}_n$ by

\begin{equation}\label{canonical}
\Psi(h)=\sum_{\alpha\models n}\zeta_\alpha(h)M_\alpha.
\end{equation} The morphism $\Psi$ we call the {\it canonical morphism} of the
combinatorial Hopf algebra $(\mathcal{H}, \zeta)$.

Given a composition $\alpha=(a_1,\ldots,a_k)\models n$, let
$s(\alpha)=\{a_1,\ldots,a_k\}\vdash n$ be the corresponding
partition. The Hopf algebra of symmetric functions $Sym$ is a Hopf
subalgebra of $QSym$ linearly spanned by monomial symmetric
functions $m_\lambda=\sum_{s(\alpha)=\lambda}M_\alpha$, where
$\lambda$ runs over all partitions. The {\it principal
specialization} of symmetric functions is an assignment of the
value at $x_1=x_2=\cdots=x_m=1, x_{m+1}=x_{m+2}=\cdots=0$ to a
symmetric function $\phi\in Sym$. It is a polynomial in $m$
denoted by $\phi(1^m)$.

\section{Building sets}
\label{sec:2}

\begin{definition}\label{bs}

A collection $\mathcal{B}\subset\mathcal{P}(X)$ of non-empty
subsets of a finite set $X$ is a {\it building set} on $X$ if it
satisfies two conditions:

(B1) If $S, S'\in \mathcal{B}$ and $S\cap S'\neq\emptyset$ then
$S\cup S'\in \mathcal{B}$

(B2) $\{i\}\in \mathcal{B}$ for all $i\in X$.

\end{definition}

We write $\mathcal{B}_X$ to indicate the ground set $X$. If we do
not require the condition (B2), a family $\mathcal{B}$ is called a
Whitney system on $X$ $\cite{WS2}$. Building sets on $X$ are
ordered by inclusion. The minimal building set contains only
singletons $\mathcal{D}_X=\{\{i\}|i\in X\}.$ We call
$\mathcal{D}_X$ the discrete building set on $X$. The maximal
building set $\mathcal{P}_X=\mathcal{P}(X)\setminus\{\emptyset\}$
contains all non-empty subsets of $X$. The rank
$rank(\mathcal{B})$ of a building set $\mathcal{B}$ is the
cardinality of the ground set $X$. The restriction of a building
set $\mathcal{B}$ to a subset $I\subset X$ is a building set on
$I$ defined by $\mathcal{B}|_I=\{S\in \mathcal{B}|S\subset I\}$.

Let $\mathcal{B}_X$ and $\mathcal{B}_Y$ be building sets on finite
sets $X$ and $Y$. A map $f:X\longrightarrow Y$ is a map of
building sets if $f^{-1}(S)\in\mathcal{B}_X$ for all
$S\in\mathcal{B}_Y$. We say that building sets $\mathcal{B}_X$ and
$\mathcal{B}_Y$ are equivalent if there is a bijection
$f:X\longrightarrow Y$ such that
$f(S)\in\mathcal{B}_Y\Leftrightarrow S\in\mathcal{B}_X$.

The elements of a building set are ordered by inclusion. The
restriction $\mathcal{B}|_I$ to a maximal element
$I\in\mathcal{B}$ is called a {\it connected component} of
$\mathcal{B}$. Every building set is a disjoint union of its
connected components. A building set $\mathcal{B}_X$ is {\it
connected} if $X\in\mathcal{B}_X$. The minimal connected building
set that contains $\mathcal{B}$ is
$\overline{\mathcal{B}}=\mathcal{B}\cup\{X\}$.

Suppose that is given an arbitrary collection
$\mathcal{C}\subset\mathcal{P}(X)$ of subsets of a finite set $X$,
such that every $S\in\mathcal{C}$ has at least two elements.
Define inductively the sequence of collections
\[\mathcal{C}_0=\mathcal{C}, \ \mathcal{C}_{k+1}=\mathcal{C}_k\cup\{S\cup S'| S\in\mathcal{C}_0,
S'\in\mathcal{C}_k, S\cap S'\neq\emptyset\}, \ k\geq 0.\]The union
$\mathcal{B}(\mathcal{C})=\bigcup_{k\geq0}\mathcal{C}_k$, with all
singletons $\{x\},x\in X$ added, is a building set on $X$. For a
building set $\mathcal{B}$ on $X$ there is a unique minimal
collection $\mathcal{C}\subset\mathcal{P}(X)$ such that
$\mathcal{B}=\mathcal{B}(\mathcal{C})$. We call such collection
the {\it generating collection} of a building set $\mathcal{B}$.

The motivating example of a building set comes from graph theory.

\begin{example}\label{graphical}

The graph $\Gamma=(V,E)$, with the sets of vertices $V$ and edges
$E$, is called simple if there are no either multiple edges nor
loops. For a simple graph $\Gamma=(V,E)$ a collection
$\mathcal{B}(\Gamma)=\{I\subset V| \Gamma|_I$ {\rm is connected}\}
is a building set on $V$. We call $\mathcal{B}(\Gamma)$ the {\it
graphical building set} corresponding to the graph $\Gamma$. Note
that the graphical building set $\mathcal{B}(\Gamma)$ is connected
if and only if the graph $\Gamma$ is connected. Also for each
subset of vertices $I\subset V$ the restriction
$\mathcal{B}(\Gamma)|_I$ is the graphical building set
$\mathcal{B}(\Gamma|_I)$ corresponding to the induced subgraph
$\Gamma|_I$. The set of edges $E$ is the generating collection of
the graphical building set $\mathcal{B}(\Gamma).$

\end{example}

Let $BSet$ be the vector space over the field $\mathbb{K}$ of
characteristic zero, spanned by all equivalence classes of
building sets and $BSet_n$ its subspace spanned by equivalence
classes of building sets of rank $n$. The space $BSet$, endowed
with product
$\mathcal{B}_X\cdot\mathcal{B}_Y=\mathcal{B}_X\sqcup\mathcal{B}_Y,$
where $\mathcal{B}_X\sqcup\mathcal{B}_Y=\{S\subset X\sqcup
Y|S\in\mathcal{B}_X \ {\rm or}\ S\in\mathcal{B}_Y\}$ is a building
set on disjoint union $X\sqcup Y$, and coproduct

$$\Delta(\mathcal{B}_X)=\sum_{I\subset
X}(\mathcal{B}_X)|_I\otimes(\mathcal{B}_X)|_{I^{c}}$$ is a graded,
connected commutative and cocommutative Hopf algebra. The unit is
the building set $\mathcal{B}_\emptyset$ on the empty set. Denote
by $Conn$ the family of all equivalence classes of connected
building sets. Then, as the algebra, $BSet$ is isomorphic to the
polynomial algebra $\mathbb{K}[Conn]$ generated by the family
$Conn$. As a graded, connected bialgebra, $BSet$ possesses the
antipode $S:BSet\rightarrow BSet$, determined by

\begin{equation}\label{antipode}
S(\mathcal{B}_X)=\sum_{k\geq 1}(-1)^{k}\sum_{J_1\sqcup\ldots\sqcup
J_k=X}\prod_{j=1}^{k}(\mathcal{B}_X)|_{J_j},
\end{equation}
for building sets on $X\neq\emptyset,$ where the inner sum is over
all ordered k-tuples $(J_1,\ldots,J_k)$ of non-empty disjoint
subsets, whose union is equal to $X$.

\section{Chromatic symmetric function of a building set}
\label{sec:4}

Let $\zeta$ be a character on the Hopf algebra of building sets
$BSet$ defined by
$$\zeta(\mathcal{B})=\left\{\begin{array}{cc}1, &
\mathcal{B} \ \mbox{is discrete}\\ 0, &
\mbox{otherwise}\end{array}\right..$$ For a building set
$\mathcal{B}_X$ on the set $X$ with $n$ elements and a composition
$\alpha=(a_1,\ldots,a_k)\models n$, the value of the convolution
product $\zeta_\alpha(\mathcal{B}_X),$ defined by
$(\ref{convolution})$, is the number of ordered decompositions
$X=J_1\sqcup\ldots\sqcup J_k$ such that $(\mathcal{B}_X)|_{J_i}$
is discrete and $|J_i|=a_i$, for all $i=1,\ldots,k$. We call a
function $f:X\rightarrow \mathbb{N}$ a {\it proper coloring} \ of
a building set $\mathcal{B}_X$ if for every set
$S\in\mathcal{B}_X$ with at least two elements, there are $i,j\in
S$ such that $f(i)\neq f(j)$. For each ordered decomposition
$X=J_1\sqcup\ldots \sqcup J_k$ such that $(\mathcal{B}_X)|_{J_i}$
is discrete for all $i=1,\ldots k$ and positive integers
$n_1<\cdots<n_k,$ there is a proper coloring $f$ given by
$f|_{J_i}=n_i$. Conversely, each proper coloring
$f:X\rightarrow\mathbb{N}$ of the building set $\mathcal{B}_X$,
with $f(X)=\{n_1<\cdots<n_k\}$, defines an ordered decomposition
$X=f^{-1}(\{n_1\})\sqcup\ldots f^{-1}(\{n_k\}),$ where
$(\mathcal{B}_X)|_{f^{-1}(\{n_i\})}$ is discrete for all
$i=1,\ldots k$.

\begin{definition}\label{csf}

The {\it chromatic symmetric function} of a building set
$\mathcal{B}_X\in BSet_n$ is its image under the canonical
morphism $\Psi:(BSet,\zeta)\rightarrow(QSym,\zeta_Q)$, given by
$(\ref{canonical})$ with

\[\Psi(\mathcal{B}_X)=\sum_{\alpha\models
n}\zeta_\alpha(\mathcal{B}_X)M_\alpha.\]

\end{definition}

\noindent The function $\Psi(\mathcal{B}_X)$ is obviously
symmetric. The principal specialization of the function
$\Psi(\mathcal{B}_X)$ counts proper colorings with finite number
of colors.

\begin{definition}\label{cp}

The {\it chromatic polynomial} $\chi(\mathcal{B}_X,m)$ of a
building set $\mathcal{B}_X$ is the principal specialization

\[\chi(\mathcal{B}_X,m)=\Psi(\mathcal{B}_X)(1^m).\]

\end{definition}

\noindent Since the principal specialization on monomial basis is
given by $M_\alpha(1^m)={m \choose k(\alpha)}$, where $k(\alpha)$
is the length of a composition $\alpha\models n$, we obtain

\[\chi(\mathcal{B}_X,m)=\sum_{\alpha\models
n}\zeta_\alpha(\mathcal{B}_X){m\choose k(\alpha)}.\] We are
especially interested in the value of the chromatic polynomial of
a building set at $m=-1$, which is

\begin{equation}\label{invariant}
\chi(\mathcal{B}_X,-1)=\sum_{\alpha\models
n}(-1)^{k(\alpha)}\zeta_\alpha(\mathcal{B}_X).
\end{equation} It defines some numerical invariant,
which we call the {\it $(-1)-$invariant} of building sets.

Recall the definition of the Hopf algebra of simple graphs,
considered by Schmitt, $\cite{WS1}$. Let $\mathcal{G}$ be the
$\mathbb{K}-$vector space spanned by all equivalence classes of
finite simple graphs, graded by the number of vertices of a graph.
The space $\mathcal{G}$ is a graded, commutative and cocommutative
Hopf algebra with product
$\Gamma_1\cdot\Gamma_2=\Gamma_1\sqcup\Gamma_2$ (the disjoint union
of graphs) and coproduct $$\Delta(\Gamma)=\sum_{I\subset
V}\Gamma|_I\otimes\Gamma|_{I^{c}},$$ where $V$ is the set of
vertices of a graph $\Gamma$ and $\Gamma|_I$ its restriction on
vertices $I\subset V$. Let
$\zeta_{\mathcal{G}}:\mathcal{G}\rightarrow\mathbb{K}$ be
$$\zeta_{\mathcal{G}}(\Gamma)=\left\{\begin{array}{cc}1, & \Gamma \
\mbox{is discrete} \\ 0, & \mbox{otherwise}\end{array}\right..$$
The canonical morphism $\Psi_{\mathcal{G}}:(\mathcal{G},
\zeta_{\mathcal{G}})\rightarrow(QSym,\zeta_Q)$, given by
$(\ref{canonical})$ with

\[\Psi_{\mathcal{G}}(\Gamma)=\sum_{\alpha\models
n}(\zeta_{\mathcal{G}})_\alpha(\Gamma)M_\alpha, \
\Gamma\in\mathcal{G}_n\] is Stanley's chromatic symmetric function
of a graph, constructed in $\cite{S}$. Its principal
specialization produces the chromatic polynomial
$\chi(\Gamma,m)=\Psi_{\mathcal{G}}(\Gamma)(1^m)$ of a graph. Let
$\beta:\mathcal{G}\rightarrow BSet$ be the map that sends a graph
$\Gamma$ to the corresponding graphical building set
$\mathcal{B}(\Gamma)$.

\begin{theorem}\label{mono}

The map $\beta:\mathcal{G}\rightarrow BSet$ is a monomorphism of
combinatorial Hopf algebras such that
$\Psi\circ\beta=\Psi_{\mathcal{G}}$.

\end{theorem}

\begin{proof}

First we show that $\beta$ is an algebra morphism. Indeed,
$\beta(\Gamma_1\cdot\Gamma_2)$ is the graphical building set
corresponding to the disjoint union $\Gamma_1\sqcup\Gamma_2$. It
contains all subsets $S\subset V(\Gamma_1)\sqcup V(\Gamma_2)$ such
that the restriction $(\Gamma_1\sqcup\Gamma_2)|_S$ is connected.
It is exactly the product of graphical building sets
$\mathcal{B}(\Gamma_1)\cdot\mathcal{B}(\Gamma_2)$. Recall that
$\mathcal{B}(\Gamma|_I)=\mathcal{B}(\Gamma)|_I$. Therefore,

$$\Delta(\beta(\Gamma))=\sum_{I\subset
V(\Gamma)}\mathcal{B}(\Gamma)|_I\otimes\mathcal{B}(\Gamma)|_{I^{c}}=\sum_{I\subset
V(\Gamma)}\mathcal{B}(\Gamma|_I)\otimes\mathcal{B}(\Gamma|_{I^{c}})=(\beta\otimes\beta)(\Delta(\Gamma)).$$
\noindent Thus, the map $\beta$ is a coalgebra morphism. Since the
graphical building set $\mathcal{B}(\Gamma)$ is discrete if and
only if the graph $\Gamma$ is discrete, we have
$\zeta\circ\beta=\zeta_{\mathcal{G}}$. The correspondence of
graphs and graphical building sets is bijective, so the map
$\beta$ is a monomorphism.

\end{proof}

It follows from Theorem $\ref{mono}$ that chromatic polynomials of
graphical building sets and chromatic polynomials of graphs
coincide

\[\chi(\mathcal{B}(\Gamma),m)=\chi(\Gamma,m).\] A classical
theorem of Stanley asserts that evaluating the chromatic
polynomial of a graph at $-1$ gives the number of acyclic
orientation $\cite{SSS}$. Therefore the $(-1)-$invariant
$(\ref{invariant})$ is a generalization of the number of acyclic
orientations to building sets.

An arbitrary collection of subsets
$\mathcal{H}\subset\mathcal{P}(X)$ of the ground set $X$ is called
a {\it hypergraph} on $X$. A proper coloring of a hypergraph
$\mathcal{H}$ is a map $f:X\rightarrow\mathbb{N}$, such that for
any $S\in\mathcal{H}$ with at least two elements, $f$ is not
monochromatic on $S$. The chromatic symmetric function of a
hypergraph is defined as

\[\Psi(\mathcal{H})=\sum_{{\it proper}
f:X\rightarrow\mathbb{N}}\prod_{i\in X}x_{f(i)}.\] It depends only
on minimal elements of a hypergraph, which form an antichain in
the boolean poset $\mathcal{P}(X)$. The specialization

\[\chi(\mathcal{H},m)=\Psi(\mathcal{H})(1^m)\] is the chromatic
polynomial, which counts the number of proper colorings of
$\mathcal{H}$ with $m$ colors.

The following theorem, which is a simple consequence of
definitions, shows in what extent the chromatic symmetric function
distinguishes building sets.

\begin{theorem}\label{hypergraph}

Let $\mathcal{C}_{\min}$ be the collection of minimal elements of
the generating collection $\mathcal{C}$ of a building set
$\mathcal{B}(\mathcal{C})$. Then

\[\Psi(\mathcal{B}(\mathcal{C}))=\Psi(\mathcal{C}_{\min}),\]where
$\Psi(\mathcal{C}_{\min})$ is the chromatic symmetric function of
the hypergraph $\mathcal{C}_{\min}$.

\end{theorem}

\begin{proof}

A coloring $f:X\rightarrow\mathbb{N}$ is a proper coloring of
$\mathcal{B}(\mathcal{C})$ if and only if $f$ is not monochromatic
on an arbitrary $S\in\mathcal{C}_{\min}$. Thus, the building set
$\mathcal{B}(\mathcal{C})$ and the hypergraph $\mathcal{C}_{\min}$
have the same sets of proper colorings, so their chromatic
symmetric functions are equal.

\end{proof}

\noindent Define the {\it minimalization} of a building set
$\mathcal{B}=\mathcal{B}(\mathcal{C})$ as the building set
$\check{\mathcal{B}}=\mathcal{B}(\mathcal{C}_{\min})$. By Theorem
$\ref{hypergraph}$, building sets $\mathcal{B}(\mathcal{C}^{'})$
and $\mathcal{B}(\mathcal{C}^{''})$, with the same
minimalizations, have the same chromatic symmetric functions

\[\Psi(\mathcal{B}(\mathcal{C}^{'}))=\Psi(\mathcal{B}(\mathcal{C}^{''})).\]

The fundamental property of chromatic polynomials of graphs is the
dele\-tion-contraction property

\[\chi(\Gamma,m)=\chi(\Gamma\setminus e,m)-\chi(\Gamma/e,m),\]
where $\Gamma\setminus e$ is $\Gamma$ with an edge $e\in
E(\Gamma)$ deleted and $\Gamma/e$ is $\Gamma$ with $e$ contracted
to a point. The deletion-contraction recurrence was proved for the
chromatic polynomial of a hypergraph in $\cite{GW}$. We state the
deletion-contraction property in the setting of building sets. Let
$S\in\mathcal{C}_{\min}$ be a minimal element of the generating
collection $\mathcal{C}$ of a building set
$\mathcal{B}=\mathcal{B}(\mathcal{C})$ on the ground set $X$. The
{\it deletion} is a building set
$\mathcal{B}(\mathcal{C}\setminus\{S\})$ on $X$, generated by the
collection $\mathcal{C}\setminus\{S\}$. We denote it by
$\mathcal{B}\setminus S$ without confusing with the set-theoretic
deletion. Denote by $X/S=X\setminus S\cup\{S\}$. For a subset
$A\subset X$ denote
by $A/S=\left\{\begin{array}{ll} A, & A\cap S=\emptyset \\
A\setminus S\cup\{S\}, & A\cap S\neq\emptyset\end{array}\right.,$
which is a subset of $X/S$. The {\it contraction} is a building
set $\mathcal{B}/S=\mathcal{B}(\mathcal{C}/S)$ on $X/S$, generated
by the collection $\mathcal{C}/S=\{A/S|A\in\mathcal{C}\}$.

\begin{figure}[!h]
  \begin{center}
    \includegraphics[width=0.7\textwidth]{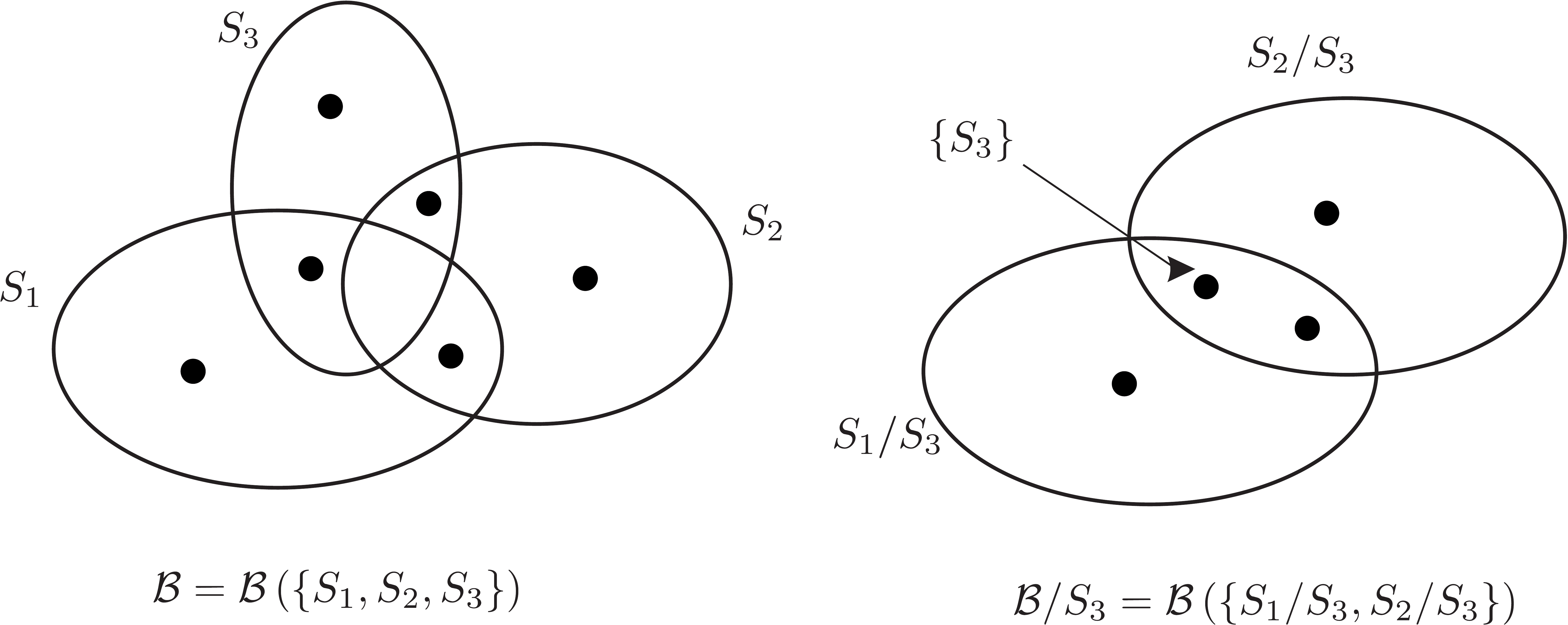}
  \end{center}
  \caption{\label{contr} The contraction of a building set}
\end{figure}

\begin{theorem}\label{del-contr}

Let $S\in\mathcal{C}_{\min}$ be a minimal element of the
generating collection $\mathcal{C}$ of a building set
$\mathcal{B}=\mathcal{B}(\mathcal{C})$. Then chromatic polynomials
of the building set $\mathcal{B}$ and its deletion
$\mathcal{B}\setminus S$ and contraction $\mathcal{B}/S$ are
related by

\[\chi(\mathcal{B},m)=\chi(\mathcal{B}\setminus S,m)-\chi(\mathcal{B}/S,m).\]

\end{theorem}

\begin{proof}

Every proper coloring $f:X\rightarrow[m]$ of $\mathcal{B}$ is a
proper coloring of $\mathcal{B}\setminus S$. A proper coloring
$f:X\rightarrow[m]$ of $\mathcal{B}\setminus S$ is a proper
coloring of $\mathcal{B}$ if and only if $f(i)\neq f(j)$ for some
$i,j\in S$. The formula follows from the fact that the set of
proper colorings $f:X\rightarrow[m]$ of $\mathcal{B}\setminus S$
which are monochromatic on $S$ and the set of all proper colorings
of $\mathcal{B}/S$ have the same number of elements.

\end{proof}

\begin{definition}\label{fs}
Let $\mathcal{L}$ be a finite antichain of nonempty finite sets.
We say that a set $S\in\mathcal{L}$ is a {\it free set} of
$\mathcal{L}$ if $S\cap\cup(\mathcal{L}\setminus\{S\})$ has at
most one element.
\end{definition}

\noindent The next proposition is an immediate consequence of the
deletion-contra\-ction property of chromatic polynomials of
building sets.

\begin{proposition}\label{free}
Given a building set $\mathcal{B}=\mathcal{B}(\mathcal{C})$ and
its minimalization
$\check{\mathcal{B}}=\mathcal{B}(\mathcal{C}_{\min})$. If $S
\in\mathcal{C}_{\min}$ is a free set of $\mathcal{C}_{\min}$ then

(i) \ $\chi(\mathcal{B},m)=\chi(\check{\mathcal{B}}\setminus
S,m)(m^{|S|}-m)$ if
$S\cap\cup(\mathcal{C}_{\min}\setminus\{S\})=\emptyset$

(ii) $\chi(\mathcal{B},m)=\chi(\check{\mathcal{B}}\setminus
S,m)(m^{|S|-1}-1)$ if
$S\cap\cup(\mathcal{C}_{\min}\setminus\{S\})$ has one element.

\end{proposition}

In the following two theorems are given expansions of chromatic
symmetric functions of building sets in the power sum basis of
symmetric functions. These expansions are just restatements in the
setting of building sets of Stanley's expansions of chromatic
symmetric functions of hypergraphs $\cite{SS}$. They are analogues
to the power sum expansions of chromatic symmetric functions of
graphs $\cite{S}$. The derived formulas for the chromatic
polynomial of a building set extend the classical Whitney's
formulas for the chromatic polynomial of a graph $\cite{W}$.

\begin{theorem}\label{powersumm1}{\rm ($\cite{SS}$, \ Theorem \ 3.5)} Let $\mathcal{C}_{\min}$
be the collection of minimal elements of the generating collection
$\mathcal{C}$ of a building set
$\mathcal{B}=\mathcal{B}(\mathcal{C})$ on the ground set $X$. For
a subcollection $\mathcal{S}\subset\mathcal{C}_{\min}$, let
$\lambda(\mathcal{S})$ be the partition of $rank(\mathcal{B})$
whose parts are equal to the ranks of connected components of the
building set $\mathcal{B}(\mathcal{S})$ on $X$. Then

\[\Psi(\mathcal{B})=\sum_{\mathcal{S}\subset\mathcal{C}_{\min}}(-1)^{|\mathcal{S}|}p_{\lambda(\mathcal{S})}.\]

\end{theorem}

The principal specializations of the power sum symmetric functions
are given by $p_{\lambda}(1^m)=m^{|\lambda|},$ where $|\lambda|$
is the size of a partition $\lambda$. For a subcollection
$\mathcal{S}\subset\mathcal{C}_{\min}$ denote by
$c(\mathcal{S})=|\lambda(\mathcal{S})|$, the number of connected
components of the building set $\mathcal{B}(\mathcal{S})$ on the
ground set $X$. We obtain the following formula for the chromatic
polynomials of building sets

\[\chi(\mathcal{B},m)=\sum_{\mathcal{S}\subset\mathcal{C}_{\min}}(-1)^{|\mathcal{S}|}m^{c(\mathcal{S})},\]

\noindent which is also a direct consequence of the
deletion-contraction property. It gives the following
interpretation of the $(-1)-$invariant of building sets
\begin{equation}\label{invariant1}
\chi(\mathcal{B},-1)=\sum_{\mathcal{S}\subset\mathcal{C}_{\min}}(-1)^{|\mathcal{S}|+c(\mathcal{S})}.
\end{equation}

Let $\mathcal{L}=\{S_1,\ldots,S_m\}$ be an antichain of the
boolean poset $\mathcal{P}(X)$. Denote by
$\mathcal{L}_I=\{S_i|i\in I\}$ the subcollection of $\mathcal{L}$
determined by a subset $I\subset[m]$. The {\it intersection poset}
$P(\mathcal{L})$ of the collection $\mathcal{L}$ is the set
$P(\mathcal{L})=\{I\subset[m]|\cap\mathcal{L}_I\neq\emptyset\}$
ordered by inclusion.

\begin{proposition}\label{parity}

Let $\mathcal{L}^{'}=\{S_1^{'},\ldots,S_m^{'}\}$ and
$\mathcal{L}^{''}=\{S_1^{''},\ldots,S_m^{''}\}$ be antichains of
finite sets with the same intersection poset
$P(\mathcal{L}^{'})=P(\mathcal{L}^{''})$. If
$|\cap\mathcal{L}^{'}_I|=|\cap\mathcal{L}^{''}_I| \ \mbox{mod} \
2$ for all $I\in P$ then

\[\chi(\mathcal{B}(\mathcal{L}^{'}),-1)=\chi(\mathcal{B}(\mathcal{L}^{''}),-1).\]

\end{proposition}

\begin{proof}

Given an antichain of finite sets
$\mathcal{L}=\{S_1,\ldots,S_m\}$, denote by $S_I=\cap_{i\in I}S_i$
and $X_I=\cup_{i\in I}S_i$ for all $I\subset[m]$. Let $c(I)$ and
$c_I$ be the numbers of connected components of building sets
$\mathcal{B}(\mathcal{L}_I)$ on the ground sets $X_{[m]}$ and
$X_I$ respectively. Since $c(I)=c_I+|X_{[m]}|-|X_I|$, by
inclusion-exclusion, we obtain
\[c(I)=c_I+|X_{[m]}|+\sum_{J\subset I}(-1)^{|J|}|S_J|.\] By formula
$(\ref{invariant1})$, we have that
$\chi(\mathcal{B}(\mathcal{L}),-1)=\sum_{I\subset[m]}(-1)^{|I|+c(I)}$,
which depends only on the intersection poset $P(\mathcal{L})$ and
the parity of $|S_I|$ for all $I\in P(\mathcal{L})$.

\end{proof}

To a building set $\mathcal{B}=\mathcal{B}(\mathcal{C})$ on the
ground set $X$ is associated the lattice $L_{\mathcal{B}}$ of
connected partitions of its minimalization
$\check{\mathcal{B}}=\mathcal{B}(\mathcal{C}_{\min})$. For a
partition $\pi=\{A_1,A_2,\ldots,A_k\}$ of $X$ let $type(\pi)$ be a
partition of $|X|$ whose components are sizes of blocks. A
partition $\pi$ of $X$ is said to be connected if the restrictions
to blocks $\check{\mathcal{B}}|_A, \ A\in\pi$ are connected as
building sets. A set of all connected partitions is ordered by
refinement of partitions, with the unique minimal element
$\widehat{0}$, which is the partition of $X$ in one-element
blocks. Denote by $|\pi|$ the number of blocks of a partition
$\pi$. Let $\mu$ be the Moebius function of a lattice
$L_\mathcal{B}$.

\begin{theorem}\label{powersumm2}{\rm ($\cite{SS}$, \ Theorem \ 3.4)} Let $L_{\mathcal{B}}$
be the lattice of connected partitions of the building set
$\mathcal{B}(\mathcal{C}_{\min})$ associated to a building set
$\mathcal{B}=\mathcal{B}(\mathcal{C})$. Then

\[\Psi(\mathcal{B})=\sum_{\pi\in
L_{\mathcal{B}}}\mu(\widehat{0},\pi)p_{type(\pi)}.\]

\end{theorem}

\noindent By the principal specialization of the power sum
symmetric functions, we obtain from Theorem $\ref{powersumm2}$ the
following interpretation of the chromatic polynomial and $(-1)-$
invariant of a building set

\[\chi(\mathcal{B},m)=\sum_{\pi\in
L_{\mathcal{B}}}\mu(\widehat{0},\pi)m^{|\pi|},\]

\[\chi(\mathcal{B},-1)=\sum_{\pi\in
L_{\mathcal{B}}}\mu(\widehat{0},\pi)(-1)^{|\pi|}.\]

\section{Symmetric functions of graphs derived from building sets}

To a simple graph $\Gamma=(V,E)$ and an integer $n\geq2$ we
associate a collection of sets $\mathcal{C}_{\Gamma,n}$ in the
following way. To an edge $e\in E$ we associate a set of new
objects $\{e_1,\ldots,e_{n-2}\}$. Define
$S_e=\{u,v,e_1,\ldots,e_{n-2}\}$, where $e=\{u,v\}\in E$ is an
edge on vertices $u,v\in V$. The collection
$\mathcal{C}_{\Gamma,n}=\{S_e|e\in E\}$ is an antichain on the
ground set $X=V\cup\{e_i | e\in E, i\in[n-2]\}$. It generates
uniquely the building set
$\mathcal{B}_{\Gamma,n}=\mathcal{B}(\mathcal{C}_{\Gamma,n})$ on
$X$, see Fig. $\ref{fig:1}$.

\begin{figure}[!h]
  \begin{center}
    \includegraphics[width=0.7\textwidth]{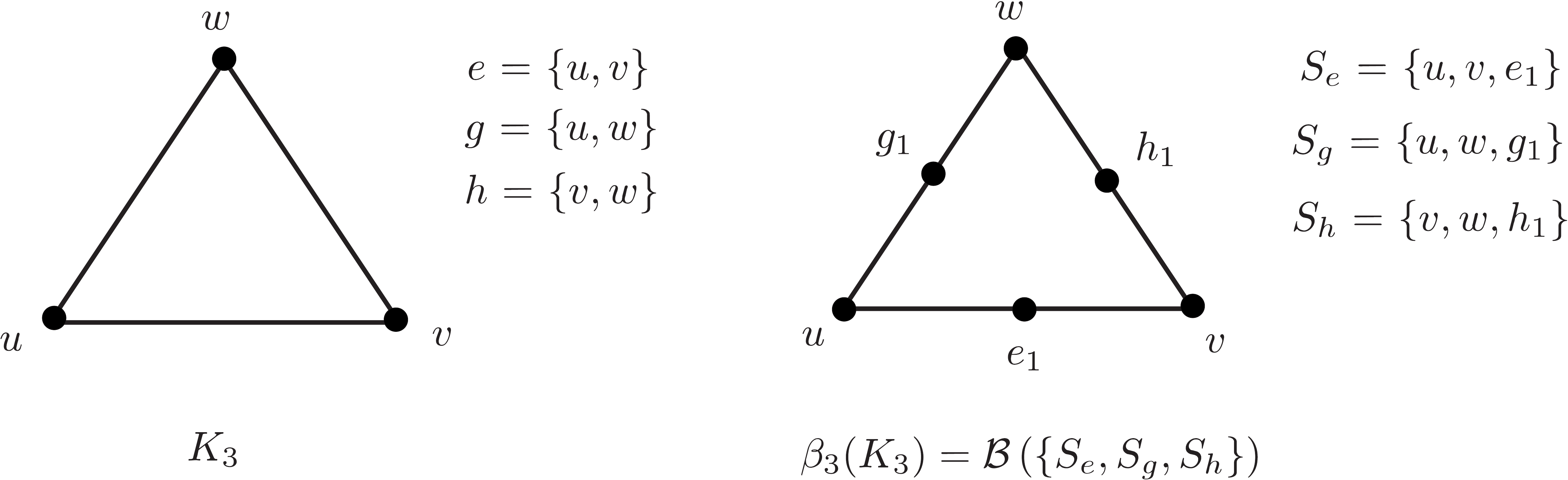}
  \end{center}
  \caption{\label{fig:1} The building set $\beta_3 (K_3)$}
\end{figure}

Define a sequence of maps $\beta_n:\mathcal{G}\rightarrow BSet, \
n\geq2$, by $\beta_n(\Gamma)=\mathcal{B}_{\Gamma,n}$, for
$\Gamma\in\mathcal{G}$. Note that $\beta_2$ is the Hopf algebra
monomorphism $\beta$ in Theorem $\ref{mono}$. It is clear from the
construction that all $\beta_n$ are algebra monomorphisms. To each
graph $\Gamma$ we associate chromatic symmetric functions
$\Psi(\beta_n(\Gamma))$ of building sets $\beta_n(\Gamma)$ and
corresponding chromatic polynomials $\chi(\beta_n(\Gamma),m)$. We
obtain a sequence of multiplicative invariants of graphs.

By Proposition $\ref{parity}$, if $n_1=n_2 \ \mbox{mod} \ 2$ then
$\chi(\beta_{n_1}(\Gamma),-1)=\chi(\beta_{n_2}(\Gamma),-1)$ for
any simple graph $\Gamma$. We obtain two numerical multiplicative
invariants of graphs, derived from $(-1)-$invariant of building
sets, namely $\chi(\beta_2(\Gamma),-1)$ and
$\chi(\beta_3(\Gamma),-1)$.

Denote by $c(S)$ the number of connected components of the
spanning subgraph $(V,S)$ with edge set $S\subset E$. The
correspondence between subsets $S\subset E$ of edges of a graph
$\Gamma$ and subcollections $\mathcal{S}=\{S_e|e\in S\}$ of the
generating collection $\mathcal{C}_{\Gamma,n}$ of the building set
$\beta_n(\Gamma)$ is bijective.  From the construction, we have

\[c(\mathcal{S})=c(S)+(n-2)(|E|-|S|),\]where $c(\mathcal{S})$ is
the number of connected components of the building set
$\mathcal{B}(\mathcal{S})$. For $n=2,3$ it follows from
$(\ref{invariant1})$ that

\[
\chi(\beta_2(\Gamma),-1)=\sum_{S\subset E}(-1)^{|S|+c(S)},
\]

\[
\chi(\beta_3(\Gamma),-1)=\sum_{S\subset E}(-1)^{|E|+c(S)}.
\]

\noindent These formulas appears to be evaluations of the Tutte
polynomial of a graph $\Gamma$
\[T_{\Gamma}(x,y)=\sum_{S\subset
E}(x-1)^{c(S)-c(E)}(y-1)^{c(S)+|S|-|V|}.\] We obtain
$\chi(\beta_2(\Gamma),-1)=(-1)^{|V|}T_{\Gamma}(2,0)$ and
$\chi(\beta_3(\Gamma),-1)=(-1)^{|E|+c(E)}T_{\Gamma}(0,2).$ The
combinatorial interpretation of these invariants is well known.
The values $T_{\Gamma}(2,0)$ and $T_{\Gamma}(0,2)$ are the numbers
of acyclic and of totally cyclic orientations of $\Gamma$ (see,
e.g. $\cite{BO}$). Recall that an oriented graph is acyclic if it
contains no directed cycles and that it is totally cyclic if every
edge is contained in some directed cycle.

\begin{example}

In $\cite{S}$, Stanley gave the example of nonisomorphic
five-vertex graphs that have the same chromatic symmetric
function, Fig. $\ref{fig:2}$. By direct calculation we obtain

\[\Psi(\beta_3(\Gamma_1))-\Psi(\beta_3(\Gamma_2))=-p_{5,3,1,1,1}+p_{6,3,1,1}+p_{7,1,1,1,1}-2p_{8,1,1,1}+2p_{9,1,1}-p_{10,1}.\]
The invariant $\chi(\beta_3(\Gamma),-1)$, derived from the
chromatic symmetric function $\Psi(\beta_3(\Gamma))$,
distinguishes those graphs.

\end{example}

\begin{figure}[!h]
  \begin{center}
    \includegraphics[width=0.6\textwidth]{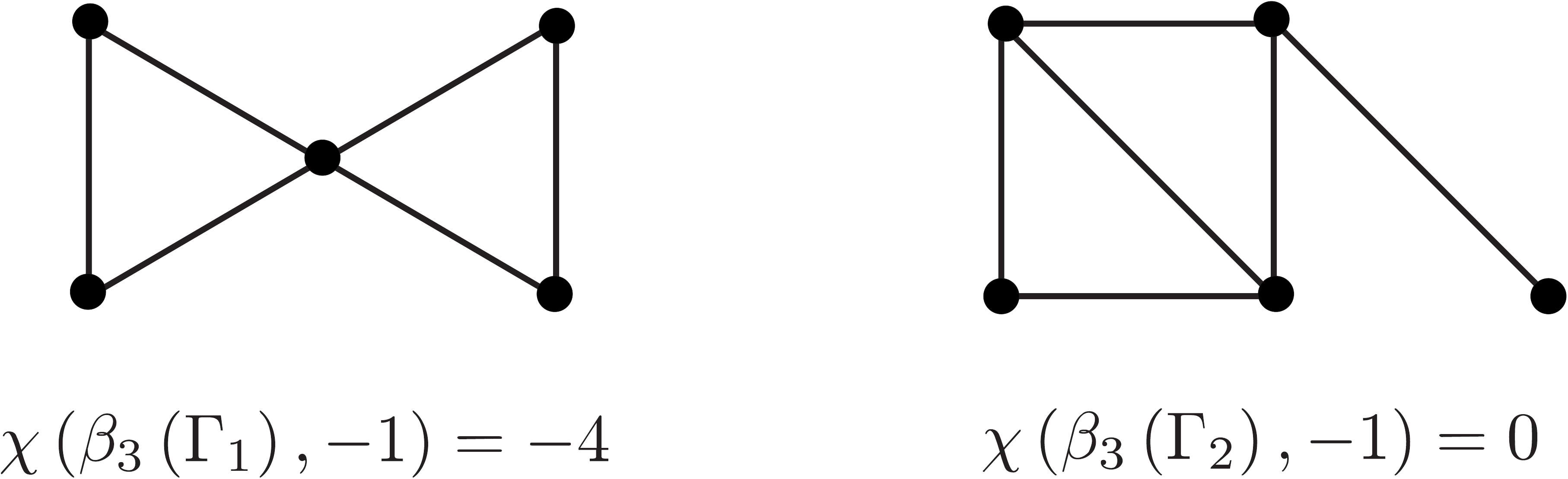}
  \end{center}
  \caption{\label{fig:2} Stanley's example of graphs with the same chromatic symmetric function}
\end{figure}

\begin{remark}

It is natural to ask in what extant chromatic symmetric functions
$\Psi(\beta_n(\Gamma))$ distinguish graphs. The previous example
shows that they are at least incomparable with the chromatic
symmetric function of a graph.

\end{remark}

\begin{remark}

The invariant $\chi(\beta_3(\Gamma),-1)$ takes values of both
signs on graphs of the same rank, in the range
$-2^{|E|}<\chi(\beta_3(\Gamma),-1)<2^{|E|}$. Hence there is not an
analogous result to the reciprocity theorem for the chromatic
polynomials of graphs, $\cite{SSSS}$.

\end{remark}

\section{Generalized Dehn-Sommerville relations for building sets}
\label{sec:5}

Let $S_{-}(\mathcal{H},\zeta)$ be the odd subalgebra of the
combinatorial Hopf algebra $(\mathcal{H},\zeta)$. It is proved in
$(\cite{ABS}, {\rm Theorem} \ 5.3.)$ that a homogeneous element
$h\in\mathcal{H}$ belongs to $S_{-}(\mathcal{H},\zeta)$ if and
only if

\begin{equation}\label{gdsr}
(id\otimes(\overline{\zeta}-\zeta^{-1})\otimes
id)\circ\Delta^{(2)}(h)=0.
\end{equation}
We refer to the previous equation as the generalized
Dehn-Sommerville relations for the combinatorial Hopf algebra
$(\mathcal{H},\zeta)$.

For the Hopf algebra of quasi-symmetric functions
$(QSym,\zeta_Q)$, these relations are described in $(\cite{ABS},
{\rm Example} \ 5.10.)$ as follows. Let $h=\sum_{\alpha\models
n}f_\alpha(h)M_\alpha$ be a homogeneous element of order
$n\in\mathbb{N}$. It satisfies the generalized Dehn-Sommerville
relations if and only if

\begin{equation}\label{bbr}
\sum_{j=0}^{a_i}(-1)^{j}f_{(a_1,\ldots,a_{i-1},j,a_i-j,a_{i+1},\ldots,a_k)}(h)=0,
\end{equation}
for each composition $\alpha=(a_1,\ldots,a_k)\models n$ and
$i\in\{1,\ldots,k(\alpha)\}$. By $k(\alpha)$ is denoted the number
of parts of the composition $\alpha$. It is understood that zero
parts in compositions are omitted. We refer to the relations
$(\ref{bbr})$ as the Bayer-Billera relations $\cite{BB}$.

Denote by $\mathcal{H}^{cop}$ the coopposite Hopf algebra of a
Hopf algebra $\mathcal{H}$ (see $\cite{DNR}$ as the general
reference for Hopf algebras). It is defined by the coproduct
$\Delta^{cop}=\tau\circ\Delta$, where
$\tau:\mathcal{H}\otimes\mathcal{H}\rightarrow\mathcal{H}\otimes\mathcal{H}$
is the twist map $\tau(x\otimes y)=y\otimes x, \
x,y\in\mathcal{H}$. Let $\mathcal{H}$ be a commutative Hopf
algebra. A map $\phi:\mathcal{H}\rightarrow\mathcal{H}$ is an
antimorphism of $\mathcal{H}$ if it is a morphism of Hopf algebras
$\phi:\mathcal{H}\rightarrow\mathcal{H}^{cop}$.

\begin{lemma}\label{l1}

Let $(\mathcal{H},\zeta)$ be a commutative combinatorial Hopf
algebra. Then the antipode
$S_\mathcal{H}:(\mathcal{H}^{cop},\zeta^{-1})\rightarrow(\mathcal{H},\zeta)$
is a morphism of combinatorial Hopf algebras.

\end{lemma}

\begin{proof}

The antipode $S_\mathcal{H}:\mathcal{H}\rightarrow\mathcal{H}$ is
an antimorphism of coalgebras. As $\mathcal{H}$ is commutative, we
have that $S_\mathcal{H}:\mathcal{H}^{cop}\rightarrow\mathcal{H}$
is a morphism of Hopf algebras. Since the inverse of a character
$\zeta$ is the composition $\zeta^{-1}=\zeta\circ S_\mathcal{H}$,
the claim follows.

\end{proof}

\begin{lemma}\label{l2}

Let $S_Q$ and $S_B$ be antipodes of Hopf algebras $QSym$ and
$BSet$ respectively. Then the following diagram is a commuting
diagram of morphisms of combinatorial Hopf algebras:

$$\begin{array}{ccc}
  (BSet^{cop},\zeta^{-1}) & \stackrel{S_B}\longrightarrow  &(BSet,\zeta)  \\
   \downarrow\Psi &   & \downarrow\Psi  \\
   (QSym^{cop},\zeta_Q^{-1})&\stackrel{S_Q}\longrightarrow  &(QSym,\zeta_Q)
\end{array}$$

\end{lemma}

\begin{proof}

The map $\Psi$ commutes with antipodes as a morphism of Hopf
algebras, so we need only prove that the above maps are morphisms
of combinatorial Hopf algebras. The Hopf algebras $QSym$ and
$BSet$ are commutative. By Lemma $\ref{l1}$, we have that both
$S_Q:(QSym^{cop},\zeta_Q^{-1})\rightarrow(QSym,\zeta_Q)$ and
$S_B:(BSet^{cop},\zeta^{-1})\rightarrow(BSet,\zeta)$ are morphisms
of combinatorial Hopf algebras. The map
$\Psi:BSet^{cop}\rightarrow QSym^{cop}$ is a morphism of Hopf
algebras. It follows from
$$\zeta_Q^{-1}\circ\Psi=\zeta_Q\circ
S_Q\circ\Psi=\zeta_Q\circ\Psi\circ S_B=\zeta\circ
S_B=\zeta^{-1},$$ that $\Psi:(BSet^{cop},
\zeta^{-1})\rightarrow(QSym^{cop},\zeta_Q^{-1})$ is a morphism of
combinatorial Hopf algebras as well.

\end{proof}

We obtain the following formula for the inverse of the character
$\zeta$ on the Hopf algebra of building sets.

\begin{proposition}\label{zeta}

The value $\zeta^{-1}(\mathcal{B}_X)$ is the $(-1)-$invariant of a
building set $\mathcal{B}_X\in BSet_n$

\[
\zeta^{-1}(\mathcal{B}_X)=\chi(\mathcal{B}_X,-1).
\]

\end{proposition}

\begin{proof}

The inverse of the universal character $\zeta_Q$ on the Hopf
algebra of quasi-symmetric functions $QSym$, calculated in
$(\cite{ABS}, {\rm Example \ 4.8.})$, is given on the monomial
basis by

$$\zeta_Q^{-1}(M_\alpha)=(-1)^{k(\alpha)}.$$ Hence, by Lemma $\ref{l2}$, we have
$$\zeta^{-1}(\mathcal{B}_X)=\zeta_Q^{-1}\circ\Psi(\mathcal{B}_X)=\zeta_Q^{-1}(\sum_{\alpha\models n}\zeta_\alpha(\mathcal{B}_X)M_\alpha)=
\sum_{\alpha\models
n}(-1)^{k(\alpha)}\zeta_\alpha(\mathcal{B}_X).$$ This is
precisely, by $(\ref{invariant})$, the $(-1)-$invariant of a
building set $\mathcal{B}_X$.

\end{proof}

\begin{example}
Let $d_n=\zeta^{-1}(\mathcal{D}_n), n\in\mathbb{N}$ be values of
$\zeta^{-1}$ on discrete building sets $\mathcal{D}_n,
n\in\mathbb{N}$. Setting $\mathcal{D}_n, n\in\mathbb{N}$ in the
identity $\zeta^{-1}\ast\zeta=\epsilon$ gives the following
recursive relation $\sum_{i=0}^{n}d_i{n \choose i}=0,
n\in\mathbb{N},$ where $d_0=\zeta^{-1}(\mathcal{B}_\emptyset)=1,$
which is uniquely satisfied by $d_n=(-1)^{n},n\in\mathbb{N}$. The
same is obtained by calculating the chromatic polynomial of a
discrete building set $\chi(\mathcal{D}_n,m)=m^n$ at $m=-1$.
Denote by ${n\choose\alpha}={n\choose a_1 \ a_2\cdots
a_k}=\frac{n!}{a_1!a_2!\cdots a_k!}$ the multinomial coefficient
corresponding to the composition
$\alpha=(a_1,a_2,\ldots,a_k)\models n$. For a discrete building
set $\mathcal{D}_n$ and a composition $\alpha\models n,$ we have
$\zeta_\alpha(\mathcal{D}_n)={n\choose \alpha}.$ We obtain the
following identity

\begin{equation}\label{identity}
\sum_{\alpha\models n}(-1)^{k(\alpha)}{n\choose\alpha}=(-1)^{n}.
\end{equation}

\end{example}

By Theorem $\ref{del-contr}$, we have that the inverse character
$\zeta^{-1}$ satisfies the deletion-contraction property. Let
$\mathcal{C}$ be the generating collection of a building set
$\mathcal{B}=\mathcal{B}(\mathcal{C})$ and
$S\in\mathcal{C}_{\min}$ be a minimal element of collection
$\mathcal{C}$. Then

\[\zeta^{-1}(\mathcal{B})=\zeta^{-1}(\mathcal{B}\setminus
S)-\zeta^{-1}(\mathcal{B}/S).\]

\noindent The following lemma is an immediate consequence of
Proposition $\ref{free}$.

\begin{lemma}\label{oddfree}

Let $\mathcal{B}=\mathcal{B}(\mathcal{C})$ be a building set such
that the minimal collection $\mathcal{C}_{\min}$ contains a free
set of odd cardinality. Then $\zeta^{-1}(\mathcal{B})=0$.

\end{lemma}

\noindent The conjugate character $\overline{\zeta}$ on $BSet$ is
given by

\begin{equation}\label{conjugate}
\overline{\zeta}(\mathcal{B}_X)=\left\{\begin{array}{cc}(-1)^{n},
& \mathcal{B}_X \ \mbox{is discrete of rank} \ n \\ 0, &
\mbox{otherwise}\end{array}\right..
\end{equation}

\noindent According to the generalized Dehn-Sommerville relations
$(\ref{gdsr})$ for the combinatorial Hopf algebra of building sets
$(BSet,\zeta)$, we have

$$\mathcal{B}_X\in S_-(BSet,\zeta) \ {\rm if \ and \ only \ if} \ (id\otimes(\zeta^{-1}-\overline{\zeta})\otimes
id)\circ\Delta^{(2)}(\mathcal{B}_X)=0.$$It follows from

$$(id\otimes(\zeta^{-1}-\overline{\zeta})\otimes
id)\circ\Delta^{(2)}(\mathcal{B}_X)=\sum_{I\sqcup J\sqcup
K=X}(\mathcal{B}_X)|_I\otimes(\zeta^{-1}-\overline{\zeta})((\mathcal{B}_X)|_J)\otimes(\mathcal{B}_X)|_K=$$
$\sum_{J\subset
X}(\zeta^{-1}-\overline{\zeta})((\mathcal{B}_X)|_J)\Delta((\mathcal{B}_X)|_{J^{c}})$,
that

\begin{equation}\label{condition}
\mathcal{B}_X\in S_-(BSet,\zeta) \ {\rm if} \
(\zeta^{-1}-\overline{\zeta})((\mathcal{B}_X)|_J)=0 \ {\rm for \
all} \ J\subset X.
\end{equation}

\begin{definition}

A building set $\mathcal{B}_X$ is an {\it eulerian building set}
if for all subsets $J\subset X$, either $(\mathcal{B}_X)|_J$ is
discrete or $\zeta^{-1}((\mathcal{B}_X)|_J)=0$.

\end{definition}

\noindent By Proposition $\ref{zeta}$ and formula
$(\ref{invariant1})$, the property of being eulerian for the
building set $\mathcal{B}=\mathcal{B}(\mathcal{C})$ depends only
on the collection of minimal elements $\mathcal{C}_{\min}$, i.e
$\mathcal{B}$ is eulerian if and only if its minimalization
$\check{\mathcal{B}}=\mathcal{B}(\mathcal{C}_{\min})$ is eulerian.
From $(\ref{conjugate})$ and $(\ref{condition})$ we obtain the
following property.

\begin{proposition}\label{eulerian}

If $\mathcal{B}_X$ is an eulerian building set then
$\mathcal{B}_X\in S_-(BSet)$.

\end{proposition}

It follows from the definition that for any eulerian building set
$\mathcal{B}_X$ and a subset $J\subset X$, the restriction
$(\mathcal{B}_X)|_J$ is an eulerian building set as well. Let
$\mathcal{E}Set$ be the subspace of $BSet$ spanned by all eulerian
building sets. Restrictions of an eulerian building set and
disjoint unions of eulerian building sets are again eulerian, so
the subspace $\mathcal{E}Set$ is a Hopf subalgebra of the Hopf
algebra of building sets $BSet$.

\begin{example}\label{discrete}

Let $\overline{\mathcal{D}}_n=\{\{1\},\{2\},\ldots,\{n\},[n]\}$.
From $(\ref{identity})$ we obtain
$\zeta^{-1}(\overline{\mathcal{D}}_n)=(-1)^n+1$. The same is a
simple consequence of the deletion-contraction property. Thus,
$\overline{\mathcal{D}}_n$ is eulerian if and only if $n$ is odd.
Let $\mathcal{C}$ be the generating collection of an eulerian
building set $\mathcal{B}=\mathcal{B}(\mathcal{C})$ and
$S\in\mathcal{C}_{\min}$ be a minimal element of the collection
$\mathcal{C}$. The restriction
$\mathcal{B}|_S=\overline{\mathcal{D}}_{|S|}$ is eulerian, so $S$
has an odd number of elements. Specially, nondiscrete graphical
building sets are not eulerian. It means that there is no
analogues notion of eulerian graphs in the combinatorial Hopf
algebra $(\mathcal{G},\zeta_{\mathcal{G}})$.

\end{example}

\begin{theorem}\label{bb}

Let $\mathcal{B}_X$ be an eulerian building set. Then
\begin{equation}\label{bbforbs}
\sum_{j=0}^{a_i}(-1)^{j}\zeta_{(a_1,\ldots,a_{i-1},j,a_i-j,a_{i+1},\ldots,a_k)}(\mathcal{B}_X)=0,
\end{equation}
where $\alpha=(a_1,\ldots,a_k)\models |X|$ and
$i\in\{1,\ldots,k(\alpha)\}$.

\end{theorem}

\begin{proof}

Let $\Psi:(BSet,\zeta)\longrightarrow (QSym,\zeta_Q)$ be the
canonical morphism of the Hopf algebra $BSet$. Since
$$\mathcal{E}Set\subset S_-(BSet,\zeta)$$ and
$$\Psi(S_-(BSet,\zeta))\subset S_-(QSym,\zeta_Q),$$ relations
$(\ref{bbforbs})$ follow from the Bayer-Billera relations
$(\ref{bbr}).$

\end{proof}

\section{The cd-index of eulerian building sets}

The $\mathbf{c}\mathbf{d}-$ index $\Phi_P(\mathbf{c},\mathbf{d})$
of an eulerian poset $P$ is a polynomial in noncommutative
variables, which efficiently encodes the flag $f-$vector. Its
existence is equivalent to the Bayer-Billera relations
($\cite{BK}$, Theorem 4.). There is a general result for the
existence of the $\mathbf{c}\mathbf{d}-$index which is known from
$\cite{A}$. The algebra of noncommutative polynomials
$\mathbb{K}<\mathbf{a},\mathbf{b}>$ with the comultiplication
defined by $\Delta(\mathbf{a})=\Delta(\mathbf{b})=1\otimes 1$ is
the terminal object in the category of infinitesimal Hopf
algebras. The canonical morphism sends the eulerian subalgebra of
an infinitesimal Hopf algebra to the subalgebra of polynomials
generated by variables $\mathbf{c}=\mathbf{a}+\mathbf{b}$ and
$\mathbf{d}=\mathbf{a}\mathbf{b}+\mathbf{b}\mathbf{a}$. In this
section we construct the $\mathbf{c}\mathbf{d}-$index of eulerian
building sets directly, by analogy with Stanley's proof of the
existence of the $\mathbf{c}\mathbf{d}-$index of eulerian posets
${\rm (\cite{Scd}, \ Theorem \ 1.1)}$.

To a composition $\alpha=(a_1,a_2,\ldots,a_k)\models n$ is
associated the set
$S(\alpha)=\{a_1,a_1+a_2,\ldots,a_1+\cdots+a_{k-1}\}\subset[n-1]$.
The compositions are ordered by $\beta\preceq\alpha$ if and only
if $S(\beta)\subset S(\alpha)$. Let $u_\alpha$ be a monomial in
two noncommutative variables $\mathbf{a}$ and $\mathbf{b}$ defined
by

\[u_\alpha=\prod_{i=1,n}u_{\alpha,i}, \ \mbox{where} \ u_{\alpha,i}=\left\{\begin{array}{cc}\mathbf{a}, &
i\notin S(\alpha)\\ \mathbf{b}, & i\in S(\alpha)
\end{array}\right..\] Given a building set $\mathcal{B}$ of the rank $n$. Define the flag $f-$vector of
$\mathcal{B}$ to be $(\zeta_\alpha(\mathcal{B}))_{\alpha\models
n}$ and the flag $h-$vector
$(\eta_\alpha(\mathcal{B}))_{\alpha\models n}$ to be a regular
linear transformation

\begin{equation}\label{h}
\eta_\alpha(\mathcal{B})=\sum_{\beta\preceq\alpha}(-1)^{k(\alpha)-k(\beta)}\zeta_\beta(\mathcal{B}).
\end{equation}
The $\mathbf{a}\mathbf{b}-$index of $\mathcal{B}$ is the
generating function of the flag $h-$vector

\[\Psi_{\mathcal{B}}(\mathbf{a},\mathbf{b})=\sum_{\alpha\models
n}\eta_\alpha(\mathcal{B})u_\alpha.\]

\begin{theorem}\label{cdindex}
If $\mathcal{B}\in\mathcal{E}Set$ is an eulerian building set then
there is a polynomial $\Phi_\mathcal{B}(\mathbf{c},\mathbf{d})$ in
variables $\mathbf{c}=\mathbf{a}+\mathbf{b}$ and
$\mathbf{d}=\mathbf{a}\mathbf{b}+\mathbf{b}\mathbf{a}$, called the
$\mathbf{c}\mathbf{d}-$index of $\mathcal{B}$, such that
\[\Psi_\mathcal{B}(\mathbf{a},\mathbf{b})=\Phi_\mathcal{B}(\mathbf{c},\mathbf{d}).\]
\end{theorem}

\begin{proof}

Let $\mathcal{B}$ be an eulerian building set of the rank $n$ on
the ground set $X$. It is an immediate consequence of $(\ref{h})$
that

\[\Psi_{\mathcal{B}}(\mathbf{a}+\mathbf{b},\mathbf{b})=\sum_{\alpha\models
n}\zeta_\alpha(\mathcal{B})u_\alpha.\] Therefore,

\begin{equation}\label{eq1}
\Psi_{\mathcal{B}}(\mathbf{a},\mathbf{b})=\sum_{k\geq1}\sum_{S_1\sqcup\ldots\sqcup
S_k=X}(\mathbf{a}-\mathbf{b})^{|S_1|-1}\mathbf{b}(\mathbf{a}-\mathbf{b})^{|S_2|-1}\mathbf{b}\cdots(\mathbf{a}-\mathbf{b})^{|S_k|-1},
\end{equation}
where the inner sum is over all ordered decompositions
$X=S_1\sqcup\ldots\sqcup S_k$, such that the restrictions
$\mathcal{B}|_{S_i}, \ i=1,\ldots,k$ are discrete. Let $I(BSet)$
be the incidence algebra of $BSet$ over the ring of noncommutative
polynomials $\mathbb{Q}\langle\mathbf{a},\mathbf{b}\rangle$.
Define functionals $f,g,h\in I(BSet)$ by
\[f(\mathcal{B})=\mathbf{a}\Psi_{\mathcal{B}}(\mathbf{a},\mathbf{b}), \ g(\mathcal{B})=\mathbf{b}\Psi_{\mathcal{B}}(\mathbf{a},\mathbf{b}),\]
\[h(\mathcal{B})=\left\{\begin{array}{cc}(\mathbf{a}-\mathbf{b})^n,
& \mathcal{B} \ \mbox{is discrete of rank} \ n \\ 0, &
\mbox{otherwise}\end{array}\right..\] From the equation
$(\ref{eq1})$ follows

\begin{equation}\label{eq2}
\Psi_{\mathcal{B}}(\mathbf{a},\mathbf{b})=\sum_{\emptyset\neq
I\subset X:\mathcal{B}|_I \
\mbox{discrete}}(\mathbf{a}-\mathbf{b})^{|I|-1}\mathbf{b}\Psi_{\mathcal{B}|_{I^c}}(\mathbf{a},\mathbf{b}),
\end{equation}
which gives $f=h\ast g$. By the formula $(\ref{antipode})$ for the
antipode $S$ of the Hopf algebra $BSet$, we obtain the inverse of
$h$

\[h^{-1}(\mathcal{B})=h\circ S(\mathcal{B})=h(\sum_{k\geq1}(-1)^k\sum_{J_1\sqcup\ldots\sqcup
J_k=X}\prod_{i=1,k}\mathcal{B}|_{J_i})=\]
\[=(\mathbf{a}-\mathbf{b})^{|X|}\sum_{k\geq1}(-1)^k|\{J_1\sqcup\ldots\sqcup
J_k=X|\mathcal{B}|_{J_i} \ \mbox{discrete}, \
i=1,k\}|=\zeta^{-1}(\mathcal{B})(\mathbf{a}-\mathbf{b})^{|X|}.\]
Hence for eulerian building sets we have

\[h^{-1}(\mathcal{B})=\left\{\begin{array}{cc}(-1)^n(\mathbf{a}-\mathbf{b})^n,
& \mathcal{B} \ \mbox{is discrete of rank} \ n \\ 0, &
\mbox{otherwise}\end{array}\right., \
\mathcal{B}\in\mathcal{E}Set.\] By definition, the restrictions
$\mathcal{B}|_I, I\subset X$ of an eulerian building set
$\mathcal{B}$ satisfy either $\zeta^{-1}(\mathcal{B}|_I)=0$ or
$\mathcal{B}|_I$ is discrete. Therefore the identity $g=h^{-1}\ast
f$ gives the following

\begin{equation}\label{eq3}
\Psi_{\mathcal{B}}(\mathbf{a},\mathbf{b})=\sum_{\emptyset\neq
I\subset X:\mathcal{B}|_I \
\mbox{discrete}}(-1)^{|I|-1}(\mathbf{a}-\mathbf{b})^{|I|-1}\mathbf{a}\Psi_{\mathcal{B}|_{I^c}}(\mathbf{a},\mathbf{b}).
\end{equation}
Summing up the equations $(\ref{eq2})$ and $(\ref{eq3})$ gives the
following recursive formula

\begin{equation}\label{recursive}
\Psi_\mathcal{B}(\mathbf{a},\mathbf{b})=\frac{1}{2}\sum_{|I| \
\mbox{odd}}(\mathbf{c}^2-2\mathbf{d})^{\frac{|I|-1}{2}}\mathbf{c}\Psi_{\mathcal{B}|_{I^c}}(\mathbf{a},\mathbf{b})
-\frac{1}{2}\sum_{|I| \
\mbox{even}}(\mathbf{c}^2-2\mathbf{d})^{\frac{|I|}{2}}\Psi_{\mathcal{B}|_{I^c}}(\mathbf{a},\mathbf{b}),
\end{equation}
where the sums are over all nonempty subsets $I\subset X$ such
that $\mathcal{B}|_I$ is discrete. Since the restrictions of
eulerian building sets are eulerian, the statement of theorem
follows by induction of the rank of $\mathcal{B}$.

\end{proof}

\noindent Let $\bar{\alpha}\models n$ be the opposite composition
of a composition $\alpha\models n$, defined by
$S(\bar{\alpha})=[n-1]\setminus S(\alpha)$. It is an immediate
corollary of Theorem $\ref{cdindex}$ that the flag $h-$vector of
an eulerian building set is symmetric

\[\eta_\alpha(\mathcal{B})=\eta_{\bar{\alpha}}(\mathcal{B}), \ \mbox{for all} \ \alpha\models n \ \mbox{and} \ \mathcal{B}\in\mathcal{E}Set_n.\]
This is equivalent to
$\Psi_\mathcal{B}(\mathbf{a},\mathbf{b})=\Psi_\mathcal{B}(\mathbf{b},\mathbf{a}),$
which is a consequence of
$\Psi_\mathcal{B}(\mathbf{a},\mathbf{b})=\Phi_\mathcal{B}(\mathbf{a}+\mathbf{b},\mathbf{a}\mathbf{b}+\mathbf{b}\mathbf{a}).$

\noindent For $j\geq1$ define $\omega(j)=\left\{\begin{array}{cc}
\frac{1}{2}(\mathbf{c}^2-2\mathbf{d})^{\frac{j-1}{2}}\mathbf{c}, &
j \
\mbox{odd} \\
-\frac{1}{2}(\mathbf{c}^2-2\mathbf{d})^{\frac{j}{2}}, & j \
\mbox{even}\end{array}\right.$ and $\delta_j=
\left\{\begin{array}{cc} 0, & j \ \mbox{even} \\
(c^2-2d)^{\frac{j-1}{2}}, & j \ \mbox{odd}\end{array}\right..$
\noindent By iterating the recursive formula $(\ref{recursive})$
we obtain more explicitly

\begin{equation}\label{omega}
\Phi_\mathcal{B}(\mathbf{c},\mathbf{d})=\sum_{k\geq1}\sum_{\alpha=(a_1,\ldots,a_k)\models
n}
\zeta_\alpha(\mathcal{B})\omega(a_1)\cdots\omega(a_{k-1})\delta_{a_k},
\ \mathcal{B}\in\mathcal{E}Set_n.
\end{equation}

\begin{example}

Let $\Phi_n=\Phi_{\mathcal{D}_n}(\mathbf{c},\mathbf{d})$ be the
$\mathbf{c}\mathbf{d}-$index of the discrete building set
$\mathcal{D}_n$. From the recursive formula $(\ref{recursive})$ we
obtain

\begin{equation}\label{andre}
\Phi_n=\frac{1}{2}\sum_{k \ \mbox{odd}}{n \choose
k}(\mathbf{c}^2-2\mathbf{d})^{\frac{k-1}{2}}\mathbf{c}\Phi_{n-k}-\frac{1}{2}\sum_{k
\ \mbox{even}}{n \choose
k}(\mathbf{c}^2-2\mathbf{d})^{\frac{k}{2}}\Phi_{n-k}+\delta_n.
\end{equation}
For instance,
\[\Phi_2=\mathbf{c}, \ \Phi_3=\mathbf{c}^2+\mathbf{d}, \
\Phi_4=\mathbf{c}^3+2(\mathbf{c}\mathbf{d}+\mathbf{d}\mathbf{c}),
\
\Phi_5=\mathbf{c}^4+3(\mathbf{c}^2\mathbf{d}+\mathbf{d}\mathbf{c}^2)+5\mathbf{c}\mathbf{d}\mathbf{c}+4\mathbf{d}^2.\]
The following numerical identity is obtained from
$(\ref{recursive})$ and $(\ref{omega})$ by calculating the
coefficient by $\mathbf{c}^{n-1}$ in $\Phi_n$

\[[\mathbf{c}^{n-1}]\Phi_n=1=\sum_{\begin{array}{c} \alpha\models n, \\
a_{k(\alpha)} \ \mbox{odd}\end{array}}{n \choose
\alpha}\frac{(-1)^{e(\alpha)}}{2^{k(\alpha)-1}}.\] By compering
$(\ref{andre})$ with the similar recurrence relation satisfied by
the $\mathbf{c}\mathbf{d}-$index $U_{B_n}(\mathbf{c},\mathbf{d})$
of boolean posets $B_n$ ${\rm (\cite{Scd}, \ Corollary \ 1.3)},$
we obtain

\[\Phi_{\mathcal{D}_n}(\mathbf{c},\mathbf{d})=U_{B_n}(\mathbf{c},\mathbf{d}),\]which
is recognized as the certain combinatorially defined polynomial,
called Andre polynomial. The generating function of $\Phi_n$ ${\rm
(\cite{Scd}, \ Corollary \ 1.4)}$ is given by

\[\sum_{n\geq1}\Phi_n\frac{x^n}{n!}=\frac{2\sinh(\mathbf{a}-\mathbf{b})x}{\mathbf{a}-\mathbf{b}}
\left(1-\frac{\mathbf{c}\sinh(\mathbf{a}-\mathbf{b})x}{\mathbf{a}-\mathbf{b}}+\cosh(\mathbf{a}-\mathbf{b})x\right)^{-1}.\]

\end{example}

\section{The geometric characterization of eulerian building sets}

In this section we give the complete characterization of the class
of eulerian building sets by some geometric conditions. To an
antichain $\mathcal{L}$ of subsets of the ground set $X$ is
associated its {\it nerve} $\Delta(\mathcal{L})$, which is an
abstract simplicial complex on the vertex set $\mathcal{L}$,
defined by
\[\Delta(\mathcal{L})=\{\mathcal{S}\subset
\mathcal{L}|\cap\mathcal{S}\neq\emptyset\}.\] Let
$|\Delta(\mathcal{L})|$ be the geometric realization of the nerve
$\Delta(\mathcal{L})$ and
link$_{\Delta(\mathcal{L})}(\{S\})=\{\mathcal{S}\in\Delta(\mathcal{L})|S\notin\mathcal{S},
(\cap\mathcal{S})\cap S\neq\emptyset\}$ and
star$_{\Delta(\mathcal{L})}(\{S\})=\{\mathcal{S}\in\Delta(\mathcal{L})|
S\in\mathcal{S}\}$ be the link and the star of a vertex $\{S\}$ in
the complex $\Delta(\mathcal{L})$.

\begin{definition}

Let $e_{\mathcal{L}}:\Delta(\mathcal{L})\rightarrow\mathbb{N}$ be
an assignment $e_{\mathcal{L}}(\mathcal{S})=|\cap\mathcal{S}|$ of
the number of elements in the intersection $\cap\mathcal{S}$ to a
simplex $\mathcal{S}\in\Delta(\mathcal{L})$. The collection
$\mathcal{L}$ is {\it odd} if $e_{\mathcal{L}}(\mathcal{S})$ is
odd for all $\mathcal{S}\in\Delta(\mathcal{L})$. A simplex
$\mathcal{S}\in\Delta(\mathcal{L})$ is {\it close} to a vertex
$S\in\mathcal{L}$ if $S^{'}\cap S\neq\emptyset$ for all
$S^{'}\in\mathcal{S}$. A simplex $\mathcal{S}$ is {\it far} from a
vertex $S$ if it is not close to $S$.

\end{definition}

\noindent The following lemma is clear from definitions.

\begin{lemma}\label{technical}

Given an element $S\in\mathcal{L}$ of an antichain $\mathcal{L}$,
then
\[\Delta(\mathcal{L}\setminus\{S\})=\Delta(\mathcal{L})\setminus
{\rm star}_{\Delta(\mathcal{L})}(\{S\}),\]

\[e_{\mathcal{L}\setminus\{S\}}(\mathcal{S})=e_{\mathcal{L}}(\mathcal{S}),
\ \mbox{for all} \
\mathcal{S}\in{\Delta(\mathcal{L}\setminus\{S\})},\]

\[\Delta(\mathcal{L}/S)=\Delta(\mathcal{L}\setminus\{S\})\cup\mathcal{P}(\{S^{'}\in\mathcal{L}|S^{'}\cap
S\neq\emptyset\}),\]

\[e_{\mathcal{L}/S}(\mathcal{S}/S)=\left\{\begin{array}{cc}
e_{\mathcal{L}}(\mathcal{S}), &  {\rm if} \ \mathcal{S} \ {\rm is
\ far \ from} \ S \\
e_{\mathcal{L}}(\mathcal{S})-e_{\mathcal{L}}(\mathcal{S}\cup\{S\})+1,
& {\rm if} \ \mathcal{S} \ {\rm is \ close \ to} \ S
\end{array}\right..\]

\end{lemma}

\begin{definition}
An antichain $\mathcal{L}=\{S_1,\ldots,S_k\}$ of subsets of $X$ is
a $k-${\it clique} if $\cap\mathcal{L}\neq\emptyset$, i.e. the
nerve $\Delta(\mathcal{L})$ is a simplex on the vertex set
$\mathcal{L}$.
\end{definition}

\begin{lemma}\label{clique}
Let $\mathcal{L}=\{S_1,\ldots,S_k\}$ be a $k-$clique on $X$. The
building set $\mathcal{B}(\mathcal{L})$ is eulerian if and only if
$\mathcal{L}$ is an odd collection.

\end{lemma}

\begin{proof}

We prove the statement by induction on $k$. For $k=1$, we have
$\mathcal{B}(\{S\})=\overline{\mathcal{D}}_{|S|}$. Hence, by
example $\ref{discrete}$, $\mathcal{B}(\{S\})$ is eulerian if and
only if $|S|$ is odd. Suppose the statement is true for any
$(k-1)-$clique. Let a $k-$clique $\mathcal{L}=\{S_1,\ldots,S_k\}$
be an odd collection and $\mathcal{B}=\mathcal{B}(\mathcal{L})$.
By the deletion-contraction property we have

\[\zeta^{-1}(\mathcal{B})=\zeta^{-1}(\mathcal{B}\setminus
S_k)-\zeta^{-1}(\mathcal{B}/S_k).\] By Lemma $\ref{technical}$,
collections $\mathcal{L}\setminus\{S_k\}$ and $\mathcal{L}/S_k$
are odd $(k-1)-$cliques, so $\zeta^{-1}(\mathcal{B})=0$ by
induction. For an arbitrary proper subset $J\subset X$ let
$I=\{i\in[k]|S_i\subset J\}, \ X_{I}=\cup_{i\in I}S_i$ and
$\mathcal{L}_{I}=\{S_i|i\in I\}$. We have

\[\mathcal{B}|_J=\mathcal{B}({\mathcal{L}_I})\sqcup\mathcal{D}_{|J|-|X_{I}|}.\]
The collections $\mathcal{L}_I, I\subset[k]$ are odd cliques, so
$\mathcal{B}$ is eulerian by induction.

Suppose that $\mathcal{L}=\{S_1,\ldots,S_k\}$ is a $k-$clique such
that $e_{\mathcal{L}}(\mathcal{S})$ is odd for all proper
subcollections $\mathcal{S}\subset\mathcal{L}$ and
$e_{\mathcal{L}}(\mathcal{L})$ is even. We say that such
collection is {\it almost odd}. The collection
$\mathcal{L}\setminus \{S_k\}$ is an odd $(k-1)-$clique. Hence by
the deletion-contraction property we have

\[\zeta^{-1}(\mathcal{B}(\mathcal{L}))=-\zeta^{-1}(\mathcal{B}(\mathcal{L})/S_k).\]
By Lemma $\ref{technical}$, the collection $\mathcal{L}/S_k$ is
almost odd. We obtain by induction on $k$ that
$\zeta^{-1}(\mathcal{B}(\mathcal{L}))=(-1)^{k-1}\zeta^{-1}(\overline{\mathcal{D}}_d),$
for some even $d$. Therefore
$\zeta^{-1}(\mathcal{B}(\mathcal{L}))=(-1)^{k-1}2$ and
$\mathcal{B}(\mathcal{L})$ is not eulerian.

Let $\mathcal{L}=\{S_1,\ldots,S_k\}$ be a $k-$clique which is
neither odd nor almost odd and $\mathcal{S}\subset\mathcal{L}$ be
a minimal almost odd subcollection. The building set
$\mathcal{B}(\mathcal{S})$ is a restriction of
$\mathcal{B}(\mathcal{L})$ such that
$\zeta^{-1}(\mathcal{B}(\mathcal{S}))\neq0$. Consequently,
$\mathcal{B}(\mathcal{L})$ is not eulerian.
\end{proof}

A simplicial complex $\Delta$ on the set of vertices $V$ is a {\it
flag complex} if for an arbitrary subset of vertices $S\subset V$,
such that $\{i,j\}\in\Delta$ for all $i,j\in S$, it follows that
$S\in\Delta$.

\begin{proposition}\label{oddflag}
Let $\mathcal{L}$ be an antichain of subsets of the ground set
$X$, such that the corresponding building set
$\mathcal{B}(\mathcal{L})$ is eulerian. Then $\mathcal{L}$ is odd
and the nerve $\Delta(\mathcal{L})$ is a flag complex.
\end{proposition}

\begin{proof}

Let $\mathcal{S}\in\Delta(\mathcal{L})$ be a clique. The building
set $\mathcal{B}(\mathcal{S})$ is eulerian as a restriction of
$\mathcal{B}(\mathcal{L})$. Hence, by Lemma $\ref{clique}$,
$\mathcal{S}$ is an odd collection. Thus
$e_{\mathcal{L}}(\mathcal{S})$ is odd for every
$\mathcal{S}\in\Delta(\mathcal{L})$, so the collection
$\mathcal{L}$ is odd.

We have to prove that all minimal non-simplices of the complex
$\Delta(\mathcal{L})$ are edges. Suppose that there is a
collection $\mathcal{S}=\{S_1,\ldots,S_k\}\subset\mathcal{L}$,
where $k>2$, which is a minimal non-simplex of
$\Delta(\mathcal{L})$. It means that $\cap\mathcal{S}=\emptyset$
and $\cap\mathcal{S}_i\neq\emptyset$ for all subcollections
$\mathcal{S}_i=\mathcal{S}\setminus\{S_i\}, \ i\in[k]$. Since
$\zeta^{-1}(\mathcal{B}(\mathcal{S}_k))=0$, it follows from the
deletion-contraction property that

\[\zeta^{-1}(\mathcal{B}(\mathcal{S}))=-\zeta^{-1}(\mathcal{B}(\mathcal{S})/S_k).\]
By Lemma $\ref{technical}$, we obtain
$e_{\mathcal{S}/S_k}(\mathcal{S}_{k}/S_k)=e_{\mathcal{S}}(\mathcal{S}_k)+1$,
which is even, and
$e_{\mathcal{S}/S_k}(\mathcal{S}^{'}/S_k)=e_{\mathcal{S}}(\mathcal{S}^{'})-e_{\mathcal{S}}(\mathcal{S}^{'}\cup\{S_k\})+1$,
for all $\mathcal{S}^{'}\subset\mathcal{S}_k$, which are odd.
Hence the clique $\mathcal{S}/S_k$ is almost odd. From the proof
of Lemma $\ref{clique}$, we have
$\zeta^{-1}(\mathcal{B}(\mathcal{S})/S_k)=(-1)^{k-2}2$ and
consequently $\zeta^{-1}(\mathcal{B}(S))=(-1)^{k-1}2$, which
contradicts the condition that $\mathcal{B}(\mathcal{L})$ is
eulerian.

\end{proof}

To an antichain $\mathcal{L}=\{S_1,\ldots,S_m\}$ is associated the
{\it intersection graph} $\Gamma(\mathcal{L})$, with the set of
vertices $V(\Gamma(\mathcal{L}))=\mathcal{L}$ and the set of edges
$E(\Gamma(\mathcal{L}))=\{\{S_i,S_j\}| S_i\neq S_j, S_i\cap
S_j\neq\emptyset\}$, which is a simple graph. To a simple graph
$\Gamma=(V,E)$ we associate an abstract simplicial complex
$\Delta(\Gamma)$, called a {\it clique complex}. A subset of
vertices $S\subset V$ is a {\it clique} of $\Gamma$ if $\{i,j\}\in
E$ for all $i,j\in S$. The clique complex is defined by
$\Delta(\Gamma)=\{S\subset V|S \ \mbox{is a clique of} \
\Gamma\}.$ We have that a flag complex is a clique complex of its
$1-$skeleton.  By Proposition $\ref{oddflag}$, for eulerian
building set $\mathcal{B}(\mathcal{L})$, the nerve
$\Delta(\mathcal{L})$ of an antichain $\mathcal{L}$ is a clique
complex $\Delta(\Gamma(\mathcal{L}))$ of the intersection graph
$\Gamma(\mathcal{L})$.

\begin{lemma}\label{cycle}

Let $\mathcal{L}$ be an odd antichain of finite sets such that the
nerve $\Delta(\mathcal{L})$ is a 1-dimensional cycle on
$n=|\mathcal{L}|$ vertices. Then
$\zeta^{-1}(\mathcal{B}(\mathcal{L}))=2(-1)^{n-1}$.

\end{lemma}

\begin{proof}
By Proposition $\ref{parity}$, we have
$\zeta^{-1}(\mathcal{B}(\mathcal{L}))=\zeta^{-1}(\beta_3(C_n))$,
where $C_n$ is the cycle graph on $n$ vertices. It follows from
the deletion-contraction property that

\[\zeta^{-1}(\beta_3(C_n))=-\zeta^{-1}(\beta_3(L_n))-\zeta^{-1}(\beta_3(C_{n-1})),\]where
$L_n$ is the path on $n$ vertices. Since
$\zeta^{-1}(\beta_3(L_n))=0$ by Lemma $\ref{oddfree}$, we obtain
$\zeta^{-1}(\beta_3(C_n))=(-1)^{n-3}\zeta^{-1}(\beta_3(C_3))$ by
induction on $n$. The direct calculation gives
$\zeta^{-1}(\beta_3(C_3))=2$.

\end{proof}

Let $\Delta$ be an abstract simplicial complex on the vertex set
$V$. A restriction $\Delta|_I$ of the complex $\Delta$ to a subset
of vertices $I\subset V$ is defined by
$\Delta|_I=\{\sigma\in\Delta|\sigma\subset I\}$. A subcomplex of
the form $\Delta|_I$ is called {\it full} subcomplex.

\begin{definition}

We say that an abstract simplicial complex $\Delta$ is {\it fully
acyclic} if it does not contain any full subcomplex which is a
1-dimensional cycle.

\end{definition}

\noindent Note that the $1-$skeleton of a fully acyclic simplicial
complex is a chordal graph. Recall that a graph $\Gamma$ is called
{\it chordal} if for each of its cycles on more than three
vertices there is an edge joining two vertices that are not
adjacent in the cycle.

\begin{proposition}\label{fully}

Let $\mathcal{L}$ be an antichain of finite sets such that the
building set $\mathcal{B}(\mathcal{L})$ is eulerian. Then the
nerve $\Delta(\mathcal{L})$ is a fully acyclic complex.
\end{proposition}

\begin{proof}

Suppose that there is a subcollection
$\mathcal{S}\subset\mathcal{L}$ such that the nerve
$\Delta(\mathcal{S})$ is a full subcomplex of
$\Delta(\mathcal{L})$ which is a 1-dimensional cycle. By
Proposition $\ref{oddflag}$, $\mathcal{S}$ is odd as a
subcollection of the odd collection $\mathcal{L}$. By Lemma
$\ref{cycle}$, we have
$\zeta^{-1}(\mathcal{B}(\mathcal{S}))\neq0$, which contradicts the
condition that $\mathcal{B}(\mathcal{L})$ is eulerian.

\end{proof}

\begin{remark}\label{chord}

A flag simplicial complex is fully acyclic if and only if its
$1-$skeleton is chordal. The class of fully acyclic flag complexes
is closed under taking restrictions of a simplicial complex.

\end{remark}

\begin{proposition}\label{opposite}

Let $\mathcal{L}$ be an odd antichain of finite sets such that the
nerve $\Delta(\mathcal{L})$ is a fully acyclic flag complex. Then,
the building set $\mathcal{B}(\mathcal{L})$ is eulerian.

\end{proposition}

\begin{proof}

We prove the statement by induction on the number of elements
$n=|\mathcal{L}|$ of an antichain $\mathcal{L}$. The statement is
true for $n=2$ by direct consideration. Suppose that the statement
is true for any antichain with at most $n-1$ elements. Let
$\mathcal{L}$ be a connected odd antichain of $n$ finite sets,
such that the nerve $\Delta(\mathcal{L})$ is a fully acyclic flag
complex. The restrictions $\Delta(\mathcal{S})$ are fully acyclic
flag complexes for all subcollection
$\mathcal{S}\subset\mathcal{L}$. Hence, by induction, we only need
to prove that $\zeta^{-1}(\mathcal{B}(\mathcal{L}))=0$. Let
$S_0\in\mathcal{L}$ be an arbitrary element. Since
$\zeta^{-1}(\mathcal{B}(\mathcal{L})\setminus S_0)=0$ by
induction, we obtain by the deletion-contraction property
\[\zeta^{-1}(\mathcal{B}(\mathcal{L}))=-\zeta^{-1}(\mathcal{B}(\mathcal{L})/S_0).\]
Let $\mathcal{S}_0=\{S\in\mathcal{L}\setminus\{S_0\}|S\cap
S_0\neq\emptyset\}$ be the collection of close vertices to the
vertex $S_0$, $\sigma=|\mathcal{P}(\mathcal{S}_0)|$ be the
geometrical simplex on vertices $\mathcal{S}_0$ and
$\Delta(\mathcal{S}_0)$ be the nerve of the collection
$\mathcal{S}_0$. By Lemma $\ref{technical}$, the geometric
realization of the nerve $\Delta(\mathcal{L}/S_0)$ is obtained as
\[|\Delta(\mathcal{L}/S_0)|=|\Delta(\mathcal{L}\setminus\{S_0\})|\cup_{|\Delta(\mathcal{S}_0)|}\sigma.\]
The nerve $\Delta(\mathcal{S}_0\cup\{S_0\})$ of the collection
$\mathcal{S}_0\cup\{S_0\}$ is a flag complex as a full subcomplex
of the flag complex $\Delta(\mathcal{L})$. Hence, by Lemma
$\ref{technical}$ and the fact that $\mathcal{L}$ is odd, we
obtain that $e_{\mathcal{L}/S_0}(\mathcal{S}/S_0)$ is odd for all
simpleces $\mathcal{S}/S_0\in\Delta(\mathcal{L}/S_0)$, where
$\mathcal{S}\subset\mathcal{L}\setminus\{S_0\}$. If
$\mathcal{S}/S_0\in\Delta(\mathcal{L}/S_0)$ is a collection which
is not in $\Delta(\mathcal{L}\setminus\{S_0\})$, we have
$e_{\mathcal{L}/S_0}(\mathcal{S}/S_0)=1$. Thus, $\mathcal{L}/S_0$
is an odd collection.

Suppose that $\mathcal{S}\subset\mathcal{L}$ is a collection such
that $\mathcal{S}/S_0$ is a minimal non-simplex of the nerve
$\Delta(\mathcal{L}/S_0)$ with at least three vertices. The
collection $\mathcal{S}$ is divided by
$\mathcal{S}=\mathcal{S}_c\cup\mathcal{S}_f$, where
$\mathcal{S}_c=\{S\in\mathcal{S}|S\cap S_0\neq\emptyset\}$ and
$\mathcal{S}_f=\{S\in\mathcal{S}|S\cap S_0=\emptyset\}$. By the
condition that $\mathcal{S}/S_0$ is a non-simplex, we have that
$\mathcal{S}_f\neq\emptyset$ and $|\mathcal{S}_c|\geq2$. We can
find two vertices $S^{'}, S^{''}\in\mathcal{S}_c$ such that
$S^{'}\cap S^{''}=\emptyset$ and a vertex $S\in\mathcal{S}_f$ such
that $S\cap S^{'}\neq\emptyset$ and $S\cap S^{''}\neq\emptyset$.
Then $\{S_0,S,S^{'},S^{''}\}$ form a full cycle in
$\Delta(\mathcal{L})$, contrary to the condition that
$\Delta(\mathcal{L})$ is fully acyclic.

Suppose that there is a collection
$\mathcal{S}\subset\mathcal{L}\setminus\{S_0\}$ such that
$\Delta(\mathcal{S}/S_0)$ is a full subcomplex of
$\Delta(\mathcal{L}/S)$ which is a 1-dimensional cycle. Then
$\Delta(\mathcal{S}\cup\{S_0\})$ is a 1-dimensional cycle which is
a full subcomplex of $\Delta(\mathcal{L})$, contrary to the
condition that $\Delta(\mathcal{L})$ is fully acyclic.

We obtain that $\Delta(\mathcal{L}/S_0)$ is fully acyclic flag
complex. By induction, $\mathcal{B}(\mathcal{L})/S_0$ is eulerian.
Hence, $\zeta^{-1}(\mathcal{B}(\mathcal{L})/S_0)=0$, which implies
$\zeta^{-1}(\mathcal{B}(\mathcal{L}))=0$.

\end{proof}

By Propositions $\ref{oddflag}$, $\ref{fully}$ and
$\ref{opposite}$ and Remark $\ref{chord}$ we obtain the complete
characterization of eulerian building sets in terms of
combinatorics of nerves of antichains of finite sets.

\begin{figure}[!h]
  \begin{center}
    \includegraphics[width=0.6\textwidth]{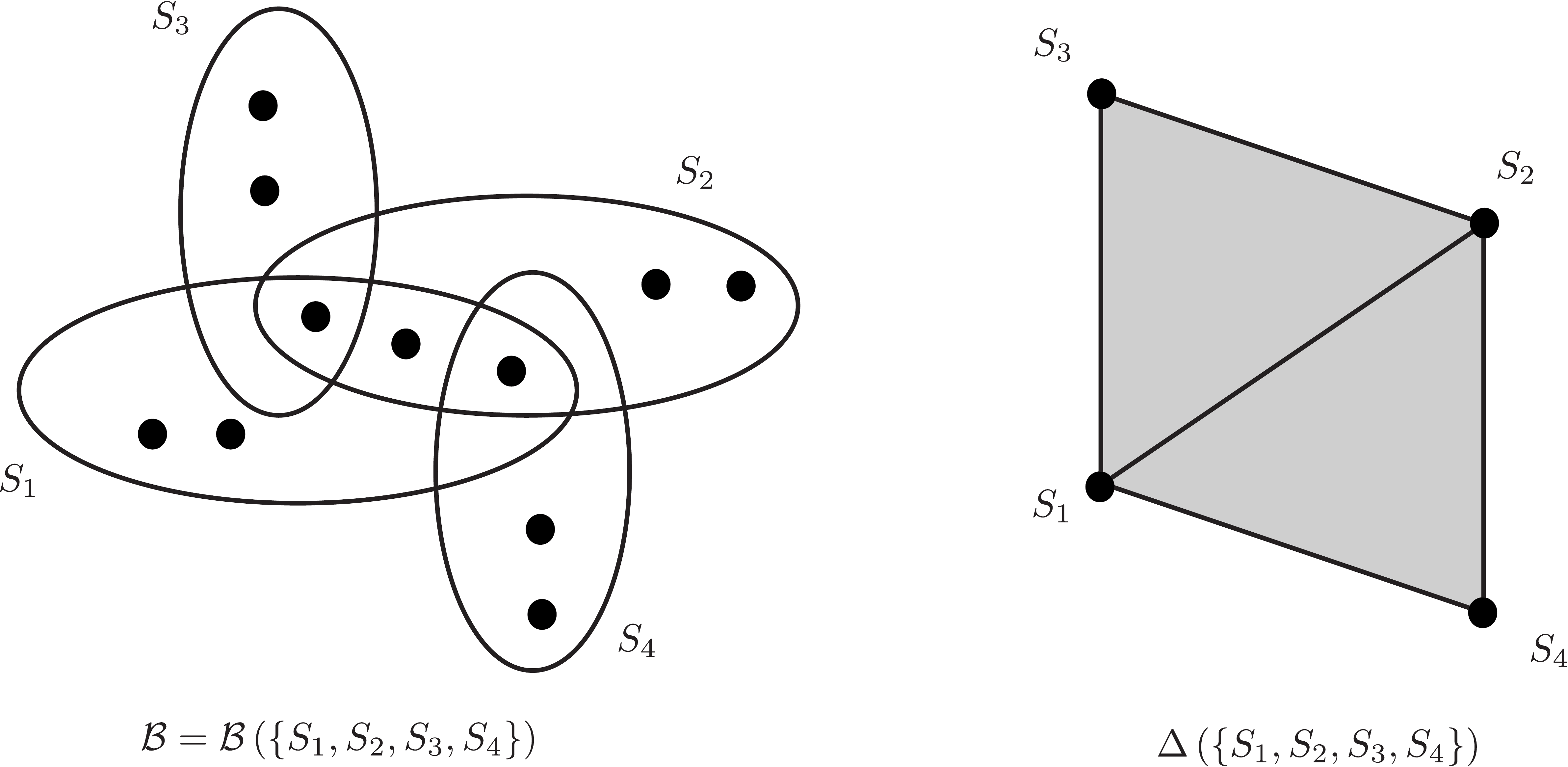}
  \end{center}
  \caption{\label{euler} The eulerian building set $\mathcal{B}\in \mathcal{E} {Set}_{11}$}
\end{figure}

\begin{theorem}\label{th}

Let $\mathcal{L}$ be an antichain of finite sets. The building set
$\mathcal{B}(\mathcal{L})$ is eulerian if and only if the
collection $\mathcal{L}$ is odd and its nerve
$\Delta(\mathcal{L})$ is the clique complex of a chordal graph.

\end{theorem}

\begin{example}

A flag 1-dimensional fully acyclic complex is a tree. Let
$\mathcal{L}$ be an antichain of finite sets such that the nerve
$\Delta(\mathcal{L})$ is a tree. By Theorem $\ref{th}$,
$\mathcal{B}(\mathcal{L})$ is eulerian if and only if the
collection $\mathcal{L}$ is odd.

\end{example}

\begin{example}
Let $\beta_{2k+1}(\Gamma)$ be a building set formed by a simple
graph $\Gamma=(V,E)$, for $k\geq1$. The nerve of the collection
$\{S_e|e\in E\}$ is the clique complex $\Delta(\Gamma^{\ast})$ of
the dual graph $\Gamma^{\ast}$, which is fully acyclic if and only
if $\Gamma$ is a tree. Thus $\beta_{2k+1}(\Gamma)$ is eulerian if
and only if $\Gamma$ is a tree.

\end{example}

\begin{remark}

The building sets produce simple polytopes called nestohedra
$\cite{FS}, \cite{P}$. The natural question to ask is how the
property of being eulerian reflects on combinatorics of nestohedra
produced from eulerian building sets.

\end{remark}

%\begin{acknowledgements}
%If you'd like to thank anyone, place your comments here
%and remove the percent signs.
%\end{acknowledgements}

\end{document}